\documentclass{amsart}

\usepackage{hyperref}
\usepackage{color}

\usepackage{amsfonts}
\usepackage{latexsym}
\usepackage{amsmath}
\usepackage{amssymb}

\newcommand{\C}{\mathbb{C}}
\newcommand{\N}{\mathbb{N}}

\renewcommand{\P}{\mathbb{P}}

\newcommand{\card}{\mathrm{card}}

\numberwithin{equation}{section}
\newtheorem{theorem}{Theorem}[section]
\newtheorem{lemma}[theorem]{Lemma}
\newtheorem{corollary}[theorem]{Corollary}
\newtheorem{definition}[theorem]{Definition}

\newtheorem{example}[theorem]{Example}

\title[Difference analogue of Cartan Second Main Theorem]{Difference analogue of Cartan's Second Main Theorem for slowly moving periodic targets}
\author[R. Korhonen, N. Li,  and K. Tohge ]
{Risto Korhonen*, Nan Li, and Kazuya Tohge}

\address{Risto Korhonen \newline
Department of Physics and Mathematics,
University of Eastern Finland,
P. O. Box 111, 80101 Joensuu, Finland.}
\email{risto.korhonen@uef.fi}

\address{Nan Li \newline
University of Jinan, School of Mathematical Sciences, Jinan, Shandong, 250022, P.R. China
\newline
and \newline
Department of Physics and Mathematics,
University of Eastern Finland,
P. O. Box 111, 80101 Joensuu, Finland.}
\email{nanli32787310@163.com}

\address{Kazuya Tohge \newline
School of Electrical and Computer Engineering,
Kanazawa University,
Kakuma-machi, Kanazawa, 920-1192, Japan.}
\email{tohge@se.kanazawa-u.ac.jp}

\thanks{This work was supported by the Academy of Finland( Grant No. 268009 \& No. 268009), the NNSF of China (Grant No. 11171013 \& No. 11371225) and the Japan Society for the Promotion of Science Grant-in-Aid for Scientific Research (C) \#22540181 and \#25400131.}
\thanks{*Corresponding Author: Risto Korhonen; Corresponding Email: risto.korhonen@uef.fi}



\begin{document}

\begin{abstract}
We extend the difference analogue of Cartan's second main theorem for the case of slowly moving periodic hyperplanes, and introduce two different natural ways to find a difference analogue of the truncated second main theorem. As applications, we obtain a new Picard type theorem and difference analogues of the deficiency relation for holomorphic curves.
\end{abstract}

\maketitle

\section{Introduction}

In 1933, Cartan \cite{cartan:33} obtained a generalization of the second main theorem to holomorphic curves. Cartan's result is a natural extension of Nevanlinna's second main theorem for the $n$-dimensional complex projective space, and it has, somewhat surprisingly, turned out to be a powerful tool for important problems in the complex plane as well. Examples of such problems appear in relation to considering Fermat-type equations, and Waring's problem for analytic functions, etc. A thorough review due to Gundersen and Hayman of the applications of Cartan's second main theorem to the complex plane can be found in \cite{gundersenh:04}. See, for instance, also \cite{lang:87,ru:01} for detailed presentations of Cartan's value distribution theory, and \cite{cherryy:01,hayman:64} for Nevanlinna theory.

Difference analogues of Cartan's second main theorem have been recently obtained, independently, by  Halburd, Korhonen and Tohge \cite{halburdkt:14}, and by Wong, Law and Wong \cite{wonglw:09}. In order to state the Cartan second main theorem for differences, we define the $n$-dimensional complex projective space $\P^n$ as the quotient space
$\bigl(\C^{n+1}\setminus\{\mathbf{0}\}\bigr)/\sim$, where
    $$(a_0,a_1, \ldots , a_{n})\sim (b_0, b_1, \ldots , b_{n})$$
if and only if
    $$(a_0,a_1, \ldots , a_{n})=\lambda (b_0, b_1, \ldots , b_{n})$$
for some $\lambda\in\C\setminus\{0\}$. The \textit{Cartan characteristic function} of a holomorphic curve
$$g=[g_0:\cdots:g_n]:\C\to\P^{n}\,,$$
or its associated system of $n+1$ entire functions $g_j$,
$$G:=(g_0,\ldots, g_n):\C\to\C^{n+1}\setminus\{\boldsymbol{0}\}\,,$$
is defined by
    \begin{equation}\label{sec1thm1.11}
    T_g(r)=T(r, G)=\int_0^{2\pi}u(re^{i\theta})\frac{d\theta}{2\pi}-u(0),
    \end{equation}
where $r>0$ and
    \begin{equation*}
    u(z)=\max_{0\leq k\leq n} \log |g_k(z)|.
    \end{equation*}
Here $g_0,\ldots, g_{n}$ are entire functions such that for all complex numbers $z$ the quantity
$\max_{0\leq k \leq n} |g_k(z)| $ is non-zero, that is, the $g_j$'s have no common zeros in the whole of~$\C$.
We call the holomorphic map~$G$ a \textit{reduced representation} of the curve $g$. The hyper-order of $g$ is defined by
\begin{eqnarray*}
  \varsigma(g)=\limsup_{r\rightarrow\infty}\frac{\log \log T_g(r)}{\log r}.
\end{eqnarray*}

Let $c\in \mathbb{C}$, and let $\mathcal{P}_{c}^{1}$ be the field of period $c$ meromorphic
functions defined in $\mathbb{C}$ of hyper-order strictly less than one. The following theorem is a difference analogue of
Cartan's result, where the ramification term has been replaced by a quantity expressed in terms of the Casorati determinant of functions which are linearly independent over a field of periodic functions.

\begin{theorem}[\cite{halburdkt:14}] \label{thm1.1}
Let $n \geq 1$, and let $g_{0},\, \ldots,\, g_{n}$ be entire functions, linearly independent over $\mathcal{P}_{c}^{1}$, such that
$\max \{|g_{0}(z)|,\, \ldots,\, |g_{n}(z)|\}> 0$ for each $z\in \mathbb{C}$, and
\begin{equation*} \label{e1.4}
\varsigma := \varsigma (g)<1, \quad  g=[g_{0}: \cdots : g_{n}].
\end{equation*}
Let $\varepsilon>0$. If $f_{0},\, \ldots,\, f_{q}$ are $q+1$ linear combinations over $\C$ of the $n+1$ functions $g_{0},\, \ldots,\, g_{n}$,
where $q>n$, such that any $n+1$ of the $q+1$ functions $f_{0},\, \ldots,\, f_{q}$ are linearly
independent over $\mathcal{P}_{c}^{1}$, and
\begin{equation*} \label{a1.1}
L=\frac{f_{0}f_{1}\cdots f_{q}}{C(g_{0},g_{1}, \ldots, g_{n})},
\end{equation*}
then
\begin{equation*}
(q-n)T_{g}(r)\leq N\left(r,\frac{1}{L}\right)-N(r,L)+o\left( \frac{T_{g}(r)}{r^{1-\varsigma -\varepsilon}} \right)+O(1),
\end{equation*}
where $r$ approaches infinity outside of an exceptional set $E$ of finite logarithmic measure (i.e. $\int _{E\cap[1,\infty)} dt/t < \infty$).
\end{theorem}

Comparing the operators $Df=f'$ and $\Delta f=f(z+1)-f(z)$, a natural
difference analogue of constant targets for $f'$ is the
periodic targets case for $\Delta f$. For instance, linear differential equations with constant coefficients can be exactly solved modulo arbitrary constants, while for linear difference equations the same statement is true but with arbitrary periodic functions. Also, as shown in \cite{halburdk:14}, the natural target space for the second main theorem in the complex plane is the solution space of
    $$
    L(f)=0,
    $$
where $L$ is a linear operator mapping a subclass $\mathcal{N}$ of the meromorphic functions in $\C$ into itself. Taking $L(f)=Df$ gives constants as targets, while the choice $L(f)=\Delta f$ yields periodic functions. Also, as in
the above Theorem~\ref{thm1.1}, the condition ``entire functions $g_{1},\, g_{2},\, \ldots,\, g_{p}$ linearly independent
over $\mathbb{C}$'' is changed naturally into ``linearly independent
over $\mathcal{P}_{c}^{1}$''. 
A natural difference analogue of Cartan's second main theorem would therefore be for slowly moving periodic target hyperplanes, rather than constants as is the case in Theorem~\ref{thm1.1}. In this paper we remedy this situation by introducing the following theorem.

\begin{theorem} \label{thm1.2}
Let $n\geq 1$, and let $g=[g_{0}: \ldots: g_{n}]$ be a holomorphic curve of $\C$ into $\P^n(\C)$ with $\varsigma := \varsigma(g)<1$, where $g_{0},\, \ldots, g_{n}$ are 
linearly independent over $\mathcal{P}_{c}^{1}$. If
    \begin{equation*}
      f_{j}=\sum_{i=0}^{n}a_{ij}g_{i} \quad  j=0,\, \ldots, \, q, \, q>n,
    \end{equation*}
where $a_{ij}$ are $c$-periodic entire functions satisfying
$T(r,a_{ij})= o(T_{g}(r))$, such that any $n+1$ of the $q+1$ functions $f_{0}, \ldots, f_{q}$
are linearly independent over $\mathcal{P}_{c}^{1}$, and
 \begin{eqnarray}\label{hab}
  L=\frac{f_{0}f_{1}\cdots f_{q}}{C(g_{0},\, g_{1},\, \ldots,\, g_{n})},
 \end{eqnarray}
then
 \begin{eqnarray}\label{hac}
   (q-n)T_{g}(r)\leq N\left(r,\frac{1}{L}\right)-N(r, L)+o(T_{g}(r)),
 \end{eqnarray}
where $r$ approaches infinity outside of an exceptional set $E$
of finite logarithmic measure.
\end{theorem}

In Section~\ref{deficiencysec} below we will show how Theorem~\ref{thm1.2} leads to two difference analogues of the truncated second main theorem, which are based on natural discrete versions of properties of the derivative function. Now, we will show that Theorem~\ref{thm1.2} implies the difference analogue of the second main theorem obtained in \cite[Theorem 2.5]{halburdk:06AASFM} in the general case of slowly moving periodic targets. (Theorem \ref{thm1.1} implies only the special case of constant targets.) To this end, let $w$ be a non-periodic meromorphic function of hyper-order $\varsigma(w)<1$. Let $g_0$ and $g_1$ be linearly independent
entire functions with no common zeros such that $w=g_0/g_1$. Let $a_j$ be $c$-periodic meromorphic functions that are small with respect to $w$ for all $j=0,\ldots,q-1$. Denote
    $$
    a_j = \frac{\alpha_j}{\beta_j}, \qquad j\in\{0,\ldots,q-1\},
    $$
where $\alpha_j$ and $\beta_j$ are $c$-periodic entire functions, and define $f_j=\beta_j g_0-\alpha_j g_1$ and $f_{q}=g_1$. Then Theorem \ref{thm1.2} yields
    \begin{equation}\label{applT}
    (q-1)T_g(r)\leq N\left(r,\frac{1}{L}\right)-N(r,L)+o(T_g(r))
    \end{equation}
where
    \begin{equation*}
    L=\frac{f_0f_1\cdots
    f_{q-1}g_{1}}{g_0\overline{g}_1-\overline{g}_0
    g_1},
    \end{equation*}
and $r$ approaches infinity outside of a set of finite logarithmic measure. The counting function $\widetilde{N}$ in the difference analogue of the second main theorem is defined in
\cite{halburdk:06AASFM} by
    \begin{equation}\label{c1}
    \widetilde{N}\left(r,\frac{1}{w-a}\right)=\int_0^r\frac{\widetilde{n}(t,a)-\widetilde{n}(0,a)}{t}\,dt+
    \widetilde{n}(0,a)\log r
    \end{equation}
where $\widetilde{n}(r,a)$ is the number of $a$-points of $w$ with multiplicity of $w(z_0)=a$ counted according to multiplicity
of $a$ at $z_0$ minus the order of zero of $\Delta_c w:=w(z+c)-w(z)$ at $z_0$. By combining \eqref{applT} with \eqref{c1}, it follows that
    \begin{equation*}
    (q-1)T(r,w)\leq \widetilde{N}(r,w)+ \sum_{j=0}^{q-1}
    \widetilde{N}\left(r,\frac{1}{w-a_j}\right)-N_0\left(r,\frac{1}{\Delta_c w}\right)
    +o(T(r,w))
    \end{equation*}
where $N_0(r,1/\Delta_c w)$ is the counting function of those zeros of
$\Delta_c w$ which do not coincide with any of the zeros or poles of $w-a_j$, and $r$ approaches infinity outside of a set of finite logarithmic measure. We have therefore shown that \cite[Theorem
2.5]{halburdk:06AASFM} follows from Theorem~\ref{thm1.2}

The remainder of the paper is organized in the following way. Section~\ref{zerossec} contains a key result (see Theorem~\ref{lem2.1} below) on linear combinations of entire functions over the field of meromorphic functions, which is a crucial tool in the proof of Theorem~\ref{thm1.2} in Section~\ref{proofsec}. Applications of the difference analogue of Cartan's theorem to Picard's theorem are in Section~\ref{picardsec}, while deficiencies and difference analogues of the truncated second main theorem can be found in Section~\ref{deficiencysec}.

\section{Zeros of linear combinations of entire functions}\label{zerossec}


One of the key problems in the proof of Theorem \ref{thm1.2} has to do with finding a lower bound for linear combinations of entire functions over the field of small functions in terms of moduli of their base functions. In the case of constant coefficients \cite[Lemma 8.2]{gundersenh:04} yields the desired results, but for non-autonomous linear combinations the situation becomes much more delicate due to possible poles and zeros of the coefficients. The key idea, which enables an applicable result needed for the proof of Theorem \ref{thm1.2}, is to formulate the estimate using a positive real valued function $A(z)$ for which the proximity function
    $$m(r,A)=\int_0^{2\pi}\log^+|A(re^{i\theta})|\frac{d\theta}{2\pi}$$
can be evaluated as $m(r,A)=o(T_g(r))$, despite of the fact that $A$ is not meromorphic. The zeros of the coefficients in the linear combinations can then be included in a small error term of the growth $o(T_g(r))$. The exact formulation is as follows.

\begin{theorem} \label{lem2.1}
Let $n\geq 1$, and let $g=[g_{0}: \ldots: g_{n}]$ be a holomorphic curve of $\C$ into $\P^n(\C)$, where $g_{0},\, \ldots, g_{n}$ are 
linearly independent over $\mathcal{P}_{c}^{1}$.
If
\begin{eqnarray*}
      f_{j}=\sum_{i=0}^{n}a_{ij}g_{i} \quad  j=0,\, \ldots, \, q, \, q>n,
\end{eqnarray*}
where $a_{ij}$ are entire functions satisfying
$T(r,a_{ij})=o(T_{g}(r))$, 
 such that
any $n+1$ of the $q+1$ functions $f_{0},\, \ldots,\, f_{q}$ are linearly independent over $\mathcal{P}_{c}^{1}$,
then there exists a positive real valued function $A(z)$, such
that
\begin{eqnarray} \label{sec1thb}
|g_{j}(z)|\leq A(z)\cdot |f_{m_{\nu}}(z)|,
\end{eqnarray}
where $0\leq j \leq n$, \, $0\leq \nu \leq q-n$, $m(r,A)=o(T_{g}(r))$ and the integers $m_{0}, \ldots, m_{q}$
are chosen so that
\begin{eqnarray}\label{eaa}
 |f_{m_{0}}(z)|\geq |f_{m_{1}}(z)|\geq \cdots \geq |f_{m_{q}}(z)|.
\end{eqnarray}
In particular, there exist at least $q-n+1$ functions $f_{j}$ that do not vanish at $z$
for all $r$ outside of a set of finite logarithmic measure, and moreover the integrated counting function $N^*(r)$ of common zeros of more than $n$ functions $f_{j}$ satisfies
    \begin{equation}\label{Nstar}
    N^*(r) = o(T_{g}(r)).
    \end{equation}
\end{theorem}

The following lemma due to G. G. Gundersen \cite{gundersen:88}, which is used to prove Theorem~\ref{lem2.1}, is an
estimation of the logarithmic measure about the total moduli of the zeros and poles
of a meromorphic function. Its statement and proof have been embedded as a part of the proof of \cite[Theorem 3]{gundersen:88}.

\begin{lemma}[\cite{gundersen:88}] \label{lem2.12}
Let $f(z)$ be a meromorphic function, and let $E_{f}:=\{r: z\in \mathbb{C}, |z|=r, f(z)=0 \; or \; f(z)=\infty\}$.
Then the set $E_{f}$ is of finite logarithmic measure.
\end{lemma}

\begin{proof}[Proof of Theorem~\ref{lem2.1}.]
The first part of the proof follows the basic idea behind the proof of \cite[Lemma 8.2]{gundersenh:04}. At the end we need to use other methods to find an estimate for the proximity function of $A(z)$, and to deal with the zeros of the coefficient functions in the linear combinations.

Since each $ f_{j}$ is a linear combination of the functions $g_{0}, \, \ldots,\, g_{n}$
with small coefficients,
     \begin{eqnarray}\label{eab}
      f_{j}=\sum_{i=0}^{n}a_{ij}g_{i} \quad  j=0,\, \ldots, \, q, \, q>n,
 \end{eqnarray}
where $T(r,a_{ij})=o(T_{g}(r))$. For each $z$, let $m_{0},\, m_{1},\, \ldots, m_{q}$
be the integers in (\ref{eaa}), which depend on $z$, and
let $\nu$ be any fixed integer satisfying $0\leq \nu \leq q-n$.
Then
\begin{eqnarray}\label{eac}
|f_{m_{\mu}}(z)|\leq |f_{m_{\nu}}(z)|,  \quad \mu=q-n,q-n+1, \ldots, q,
\end{eqnarray}
and letting $\{i_{0},\, \ldots,\, i_{n}\}\subset \{0, \ldots, q \}$, it follows from (\ref{eab})
that
\begin{eqnarray*}
f_{i_{k}}(z)= \sum _{j=0} ^{n}a_{ji_{k}}(z)g_{j}(z),  \quad k=0,1, \ldots, n,
\end{eqnarray*}
that is,
\begin{eqnarray}\label{eadd}
\left(
\begin{array}{c}
f_{i_{0}}(z)\\
f_{i_{1}}(z)\\
\vdots \\
f_{i_{n}}(z)\\
\end{array}
\right)
=
\left( \begin{array}{ccc}
a_{0i_{0}}(z) & \cdots & a_{ni_{0}}(z) \\
a_{0i_{1}}(z) & \cdots  & a_{ni_{1}}(z) \\
\vdots & \ddots & \vdots \\
a_{0i_{n}}(z) & \cdots & a_{ni_{n}}(z)
\end{array} \right)
\cdot
\left(
\begin{array}{c}
g_{0}(z)\\
g_{1}(z)\\
\vdots \\
g_{n}(z)\\
\end{array}
\right).
\end{eqnarray}
We set the determinant $d_{i_{0}\ldots i_{n}}(z)$ by
\begin{eqnarray}\label{eaeee}
d_{i_{0}\ldots i_{n}}(z):=\left|
\begin{array}{ccc}
a_{0i_{0}}(z) & \cdots & a_{ni_{0}}(z) \\
a_{0i_{1}}(z) & \cdots  & a_{ni_{1}}(z) \\
\vdots & \ddots & \vdots \\
a_{0i_{n}}(z) & \cdots & a_{ni_{n}}(z)
\end{array} \right|.
\end{eqnarray}
Since $g_{0},\, \ldots,\, g_{n}$ and
$f_{i_{0}},\, \ldots,\, f_{i_{n}}$ are linearly independent over $\mathbb{C}$
(since $\mathbb{C} \subset \mathcal{P}_{c}^{1}$),
we get $d_{i_{0}\ldots i_{n}}(z)\not\equiv 0$. Otherwise, since
$f_{i_{0}},\, \ldots,\, f_{i_{n}}$ can be expressed by using
$g_{0},\, \ldots,\, g_{n}$, from \eqref{eadd} we would have

\begin{eqnarray*}
\textrm{rank of}  \left(
\begin{array}{ccc}
a_{0i_{0}}(z) & \cdots & a_{ni_{0}}(z) \\
a_{0i_{1}}(z) & \cdots  & a_{ni_{1}}(z) \\
\vdots & \ddots & \vdots \\
a_{0i_{n}}(z) & \cdots & a_{ni_{n}}(z)
\end{array} \right)=\textrm{rank of}  \left(
\begin{array}{cccc}
a_{0i_{0}}(z) & \cdots & a_{ni_{0}}(z) & f_{i_{0}}\\
a_{0i_{1}}(z) & \cdots  & a_{ni_{1}}(z)& f_{i_{1}}\\
\vdots & \ddots & \vdots & \vdots \\
a_{0i_{n}}(z) & \cdots & a_{ni_{n}}(z) & f_{i_{n}}
\end{array} \right).
\end{eqnarray*}
But from $d_{i_{0}\ldots i_{n}}(z)\equiv 0$ and from the fact that $f_{i_{0}},\, \ldots,\, f_{i_{n}}$
are linearly independent, we have a contradiction. Thus $d_{i_{0}\ldots i_{n}}(z)\not\equiv 0$ as we claimed.

By using Cramer's rule, for each $j=0,1, \ldots, n$, we have
\begin{eqnarray*}\label{seceaee}
g_{j} &=&\frac{\left|
\begin{array}{ccccccc}
a_{0i_{0}} & \cdots  &  a_{j-1 i_{0}}  & f_{i_{0}} &  a_{j+1 i_{0}}  &  \cdots  &  a_{ni_{0}} \\
a_{0i_{1}} & \cdots  &  a_{j-1 i_{1}} & f_{i_{1}} &  a_{j+1 i_{1}} &\cdots  & a_{ni_{1}} \\
\vdots &  \cdots &  \vdots  & \vdots &\vdots &  \cdots  &\vdots \\
a_{0i_{n}} & \cdots  &  a_{j-1 i_{n}} &  f_{i_{q}}& a_{j+1 i_{n}}&\cdots  & a_{ni_{n}}
\end{array} \right|}{d_{i_{0}\ldots i_{n}}(z)}\\
&=&
(-1)^{1+j+1}\frac{\left|
\begin{array}{cccccc}
a_{0i_{1}} & \cdots  &  a_{j-1 i_{1}}  &  a_{j+1 i_{1}} &\cdots  & a_{ni_{1}} \\
\vdots &  \cdots &  \vdots  &\vdots &  \cdots  &\vdots \\
a_{0i_{n}} & \cdots  &  a_{j-1 i_{n}} & a_{j+1 i_{n}}&\cdots  & a_{ni_{n}}
\end{array} \right|}{d_{i_{0}\ldots i_{n}}(z)}\cdot f_{i_{0}}\\
&+&(-1)^{2+j+1}\frac{
\left|
\begin{array}{cccccc}
a_{0i_{0}} & \cdots  &  a_{j-1 i_{0}}  &  a_{j+1 i_{0}}  &  \cdots  &  a_{ni_{0}} \\
a_{0i_{2}} & \cdots  &  a_{j-1 i_{2}}  &  a_{j+1 i_{2}}  &  \cdots  &  a_{ni_{2}}\\
\vdots &  \cdots &  \vdots  & \vdots  &  \cdots  &\vdots \\
a_{0i_{n}} & \cdots  &  a_{j-1 i_{n}} & a_{j+1 i_{n}}&\cdots  & a_{ni_{n}}
\end{array} \right|}{d_{i_{0}\ldots i_{n}}(z)}
\cdot f_{i_{1}} \\
&+&\cdots \\
&+&
(-1)^{n+1+j+1}
\frac{\left|
\begin{array}{cccccc}
a_{0i_{0}} & \cdots  &  a_{j-1 i_{0}}  &  a_{j+1 i_{0}}  &  \cdots  &  a_{ni_{0}} \\
a_{0i_{1}} & \cdots  &  a_{j-1 i_{1}}  &  a_{j+1 i_{1}}  &  \cdots  &  a_{ni_{1}}\\
\vdots &  \cdots &  \vdots  & \vdots  &  \cdots  &\vdots \\
a_{0i_{n-1}} & \cdots  &  a_{j-1 i_{n-1}} & a_{j+1 i_{n-1}}&\cdots  & a_{ni_{n-1}}
\end{array} \right|}{d_{i_{0}\ldots i_{n}}(z)}
\cdot f_{i_{n}}.
\end{eqnarray*}
By setting
\begin{eqnarray}\label{sec2aaa}
c_{i_{0}j}(z)&:=&
(-1)^{1+j+1}\frac{\left|
\begin{array}{cccccc}
a_{0i_{1}} & \cdots  &  a_{j-1 i_{1}}  &  a_{j+1 i_{1}} &\cdots  & a_{ni_{1}} \\
\vdots &  \cdots &  \vdots  &\vdots &  \cdots  &\vdots \\
a_{0i_{n}} & \cdots  &  a_{j-1 i_{n}} & a_{j+1 i_{n}}&\cdots  & a_{ni_{n}}
\end{array} \right|}{d_{i_{0}\ldots i_{n}}(z)}      \nonumber \\
c_{i_{1}j}(z) &:=&
(-1)^{2+j+1}\frac{
\left|
\begin{array}{cccccc}
a_{0i_{0}} & \cdots  &  a_{j-1 i_{0}}  &  a_{j+1 i_{0}}  &  \cdots  &  a_{ni_{0}} \\
a_{0i_{2}} & \cdots  &  a_{j-1 i_{2}}  &  a_{j+1 i_{2}}  &  \cdots  &  a_{ni_{2}}\\
\vdots &  \cdots &  \vdots  & \vdots  &  \cdots  &\vdots \\
a_{0i_{n}} & \cdots  &  a_{j-1 i_{n}} & a_{j+1 i_{n}}&\cdots  & a_{ni_{n}}
\end{array} \right|}{d_{i_{0}\ldots i_{n}}(z)}     \\
&\vdots&    \nonumber \\
c_{i_{n}j}(z) &:=&
(-1)^{n+1+j+1}
\frac{\left|
\begin{array}{cccccc}
a_{0i_{0}} & \cdots  &  a_{j-1 i_{0}}  &  a_{j+1 i_{0}}  &  \cdots  &  a_{ni_{0}} \\
a_{0i_{1}} & \cdots  &  a_{j-1 i_{1}}  &  a_{j+1 i_{1}}  &  \cdots  &  a_{ni_{1}}\\
\vdots &  \cdots &  \vdots  & \vdots  &  \cdots  &\vdots \\
a_{0i_{n-1}} & \cdots  &  a_{j-1 i_{n-1}} & a_{j+1 i_{n-1}}&\cdots  & a_{ni_{n-1}}
\end{array} \right|}{d_{i_{0}\ldots i_{n}}(z)} \nonumber
\end{eqnarray}
it follows that
\begin{eqnarray*}\label{eae}
g_{j}(z)= \sum _{k=0} ^{n}c_{i_{k}j}(z)f_{i_{k}}(z),  \quad j=0,1, \ldots, n,
\end{eqnarray*}
where $c_{i_{k}j}$ are meromorphic functions satisfying
\begin{eqnarray*}\label{eaelem21}
T(r,c_{i_{k}j})=O\left(\sum_{i=0}^{n}\sum_{j=0}^{q}T(r,a_{ij})\right) =o(T_{g}(r)).
\end{eqnarray*}
Thus we have
\begin{eqnarray}\label{eaelem22}
|g_{j}(z)|\leq \sum _{k=0} ^{n}|c_{i_{k}j}(z)|\cdot|f_{i_{k}}(z)|,  \quad j=0,1, \ldots, n.
\end{eqnarray}
For a particular choice of $z\in \mathbb{C}$, let sequence $i_{0}\cdots i_{n}$ be $m_{q-n}\cdots m_{q}$.
Then, combining \eqref{eaelem22} with \eqref{eac} we have

\begin{eqnarray}\label{eaelem23}
|g_{j}(z)|&\leq& \sum _{k=q-n} ^{q}|c_{m_{k}j}(z)|\cdot|f_{m_{k}}(z)| \nonumber \\
&\leq& \sum _{k=q-n} ^{q}|c_{m_{k}j}(z)|\cdot|f_{m_{\nu}}(z)| \nonumber \\
&\leq& \left(\sum _{k=q-n} ^{q}|c_{m_{k}j}(z)|\right)\cdot|f_{m_{\nu}}(z)|, \nonumber \\
&\leq& \left(\sum _{j=0} ^{n}\sum _{k=q-n} ^{q}|c_{m_{k}j}(z)|\right)\cdot|f_{m_{\nu}}(z)|, \nonumber
\end{eqnarray}
where $\nu=0,\ldots, q-n$ and $j=0,1, \ldots, n$. By defining
\begin{eqnarray*}
A(z)=\sum _{j=0} ^{n}\sum _{k=q-n} ^{q}|c_{m_{k}j}(z)|,
\end{eqnarray*}
we have 
\begin{eqnarray}\label{eaelem24}
|g_{j}(z)|\leq A(z)\cdot |f_{m_{\nu}}(z)|
\end{eqnarray}
for all $\nu=0,\ldots, q-n$ and $j=0,1, \ldots, n$, where $$m(r,A(z))\leq \sum _{j=0} ^{n}\sum _{k=q-n} ^{q}m(r,c_{m_{k}j}(z))=o(T_{g}(r)).$$

Next we prove that $A(z)$ is non-zero for all $z\in \mathbb{C}$. From the assumption that $g=[g_{0}: \ldots: g_{n}]$ is a holomorphic curve of $\C$ into $\P^n(\C)$, we have that $\max \{|g_{0}(z)|,\, \ldots,\, |g_{n}(z)|\}>0$ for all $z\in \mathbb{C}$. Since $f_{i_{k}}(z)\, (k=0,\ldots, n)$ are entire
functions, it clearly follows that $|f_{m_{\nu}}(z)|$ gives a finite real number for all $z\in \mathbb{C}$. If there exists $z_{0}\in \mathbb{C}$ such that
$A(z_{0})=0$, then from \eqref{eaelem24}, we have that $|g_{j}(z_{0})|=0$ for all $j=0,1, \ldots, n$,
which contradicts with the fact that
$\max \{|g_{0}(z)|,\, \ldots,\, |g_{n}(z)|\}>0$. Thus $A(z)\not=0$ for all $z\in \mathbb{C}$.

Finally, we will prove that there exist at least $q-n+1$ functions $f_{j}$ that do not vanish at $z$
for all $r$ outside of a set of finite logarithmic measure. To this end we define the sets
\begin{equation}\label{A}
\mathfrak{A}:= \big\{z\in \mathbb{C}: d_{i_{0}\ldots i_{n}}(z)=0, 
\; \{i_{0},\ldots,i_{n}\}
\subset \{0,\ldots, q\}. \big\}
\end{equation}
and
$$E_{\mathfrak{A}}=\{r: |z|=r, z\in \mathfrak{A}\}.$$
Then we have
    $$E_{\mathfrak{A}} \subseteq \bigcup_{\{i_{0},\ldots,i_{n}\}} E_{d_{i_{0}\ldots i_{n}}} 
    $$
where $E_{d_{i_{0}\ldots i_{n}}}$
is defined as in Lemma \ref{lem2.12}. From Lemma \ref{lem2.12}, we have that
\begin{eqnarray*}\label{eaelem24aa}
\int_{E_{\mathfrak{A}}}\frac{dx}{x}\leq \int_{\bigcup_{\{i_{0},\ldots,i_{n}\}} E_{d_{i_{0}\ldots i_{n}}} 
}\frac{dx}{x}
\leq \sum_{\{i_{0},\ldots,i_{n}\}}\int_{ E_{d_{i_{0}\ldots i_{n}}}}\frac{dx}{x}
<\infty.
\end{eqnarray*}
So it follows that the set $E_{\mathfrak{A}}$ is of finite logarithmic measure. 

For all $z\in \mathbb{C}\backslash \mathfrak{A}$, we get $d_{i_{0}\ldots i_{n}}(z) \neq 0$. 
Thus from \eqref{sec2aaa}, we have that $c_{i_{k}j}(z)(k,j=0,\ldots,n)$ are analytic
on $\mathbb{C}\backslash \mathfrak{A}$. Therefore $\sum _{j=0} ^{n}\sum _{k=q-n} ^{q}|c_{m_{k}j}(z)|$ gives a finite real number for
all $z\in \mathbb{C}\backslash \mathfrak{A}$. If there exists a $z_{0}\in \mathbb{C}\backslash \mathfrak{A}$ such that $|f_{m_{\nu}}(z_{0})|=0$ for any $\nu=0,\ldots, q-n$, then from \eqref{eaelem23}, we deduce that
$|g_{j}(z_{0})|=0$ for all $j=0,1, \ldots, n$,
which contradicts with the assumption that
$$\max \{|g_{0}(z)|,\, \ldots,\, |g_{n}(z)|\}>0.$$
Thus $|f_{m_{\nu}}(z_{0})|\not=0$ for all $\nu\in\{0,\ldots, q-n\}$.

If $z_0\in \mathfrak{A}$, then there may be a zero of $f_{m_{\nu}}(z)$ at $z=z_0$ but the order of this zero is bounded by the order of the zero of $d_{i_{0}\ldots i_{n}}(z)$ at $z=z_0$. By going through all points $z\in \mathfrak{A}$, and taking into account that $T(r,d_{i_{0}\ldots i_{n}})=o(T_g(r))$, we obtain \eqref{Nstar}. This completes the proof of Theorem~\ref{lem2.1}.
\end{proof}

\section{Picard's theorem}\label{picardsec}
As an application of the difference analogue of Cartan's theorem, in \cite{halburdkt:14} Halburd, Korhonen and Tohge
obtained a difference analogue of Picard's theorem for holomorphic curves.

\begin{theorem}[\cite{halburdkt:14}] \label{thm1.31}
Let $f: \mathbb{C}\rightarrow \mathbb{P}^{n}$ be a holomorpic curve such that
$\varsigma(f)<1$, let $c\in \mathbb{C}$ and let $p\in \{1, \ldots, n+1\}$. If
$n+p$ hyperplanes in general position have
forward invariant preimages under $f$ with respect to the
translation $\tau(z)=z+c$, then the image of $f$ is contained in a projective
linear subspace over $\mathcal{P}_{c}^{1}$ of dimension $\leq [n/p]$.
\end{theorem}

Here a preimage of a hyperplane $H\subset\P^n$ under $f$ is said to be \textit{forward invariant} with respect to the translation $\tau_{c}(z)= z+c$ if
    \begin{equation}\label{finv}
    \tau_{c}(f^{-1}(\{H\}))\subset f^{-1}(\{H\})
    \end{equation}
where $f^{-1}(\{H\})$ and $\tau_{c}(f^{-1}(\{H\}))$ are multisets in which each point is repeated according to its multiplicity. Finitely many exceptional values are allowed in the inclusion \eqref{finv} if the holomorphic curve $f$ is transcendental.

As mentioned in the introduction, a natural difference analogue of Picard's theorem would have periodic moving targets. In order to state our generalization to that direction, we first need to define what do we exactly mean by a moving periodic hyperplane.

First, we fix the numbers $n$ and $q(\geq n)$, and observe $q$ moving hyperplanes $H_j(z)$ associated with
$\mathbf{a}_j=\bigl(a_{j0}(z), \ldots, a_{jn}(z)\bigr)$. Let us put $Q:=\{0, \ldots , q\}$ and $N:=\{0, \ldots , n\}$ for convenience. By $\mathcal{K}$ we denote a field containing all the $a_{jk}(z)$\, ($j\in Q,\, k\in N$) and
also $\mathbb{C}$, where $a_{jk}(z)$ are $c$-periodic entire functions.


Let $H(z)$ be an arbitrary moving hyperplane over the field $\mathcal{K}$ in $\P^{n}$, that is, a hyperplane given by
	\begin{equation} \label{thm1.31eq11}
	H(z)=\Bigl\{[x_0: \cdots : x_n] \in \P^{n} \, : \, a_0(z)x_0+ \cdots + a_n(z)x_n=0\Bigr\}\,,
	\end{equation}
where $a_0, \ldots , a_n$ are $c$-periodic entire functions. 
Thus $H(z)$ is associated with a holomorphic mapping
     $$
     \mathbf{a}(z)=\bigl(a_0(z), \ldots , a_n(z)\bigr):\C\to\C^{n+1}.
     $$
Letting $x=[x_0:\cdots:x_n]$, we denote
	$$
	L_H(x,\mathbf{a}(z))=\langle x, \mathbf{a}(z)\rangle=a_0(z)x_0+\cdots+a_n(z)x_n.
	$$
For $x=g=[g_0:\cdots:g_n]$, we then have
	$$
	L_H(g,\mathbf{a}(z))=\langle g(z), \mathbf{a}(z)\rangle=a_0(z)g_0(z)+\cdots+a_n(z)g_n(z),
	$$
and we say that the curve $g$ and the moving hyperplane $H$ is {\it free} if $L_H(g,\mathbf{a}(z))\not\equiv 0$.

\begin{definition}
Moving hyperplanes
$$
H_j(z)=\left\{[x_0: \cdots : x_n]\, :\, \sum_{i=0}^n a_{ji}(z)x_i=0\right\}
$$
in~$\P^{n}$ over $\mathcal{K}$, and  holomorphic mappings $\mathbf{a}_j(z)=\bigl(a_{j0}(z), \ldots , a_{jn}(z)\bigr)$ of~$\C$ into
$\C^{n+1}$ associated with $H_j(z)$, $j=0, \ldots , q$, are given. Let $\widetilde{\mathcal{K}}$ be a field such that $\C \subset \widetilde{\mathcal{K}}$.
We say that $H_0(z), \ldots, H_q(z)$ are \textit{in general position} over $\widetilde{{\mathcal K}}$, if $q\geq n$ and any $n+1$ of the
vectors $\mathbf{a}_j(z)$, $j=0, \ldots , q$, are linearly independent over $\widetilde{\mathcal{K}}$.
\end{definition}

In order to measure the growth of holomorphic mappings associated with moving hyperplanes, we need a modified version of the Cartan characteristic function, and the corresponding notion of hyper-order.

\begin{definition}
Let $\mathbf{a}(z)=(a_0(z), \ldots , a_n(z)):\C\to\C^{n+1}$ be a holomorphic mapping. Then
    $$
    T^*_{\mathbf{a}}(r)=\int_{0}^{2\pi}\sup_{j\in \{0,\ldots, n\}} \log |a_j(re^{i\theta})|\frac{d\theta}{2\pi}
    $$
is the \textit{characteristic function} of $\mathbf{a}$, and
    $$
    \varsigma^*(\mathbf{a})= \limsup_{r\to\infty}\frac{\log^+\log^+ T^*_{\mathbf{a}}(r)}{\log r}
    $$
is the \textit{hyper-order} of $\mathbf{a}$.
\end{definition}

We can now state our generalization of Theorem~\ref{thm1.31}.

\begin{theorem} \label{thm1.312}
Let $f: \mathbb{C}\rightarrow \mathbb{P}^{n}$ be a holomorphic curve such that
$\varsigma(f)<1$, let $c\in \mathbb{C}$, let $p\in \{1, \ldots, n+1\}$. If
$n+p$ moving $c$-periodic hyperplanes $H_{j}$ in general position over $\mathcal{P}_{c}^{1}$ with associated holomorphic mappings $\mathbf{a}_j(z)=\bigl(a_{j0}(z), \ldots , a_{jn}(z)\bigr)$ have
forward invariant preimages under $f$ with respect to the
translation $\tau(z)=z+c$, and
    \begin{equation}\label{bfa}
    \mathbf{a}_{i_{1}\cdots i_{n+2}}=(a_{i_{1}0},\ldots, a_{i_{1}n},a_{i_{2}0},\ldots, a_{i_{2}n}, \ldots, a_{i_{n+2}0},\ldots, a_{i_{n+2}n})
    \end{equation}
satisfies $\varsigma^*(\mathbf{a}_{i_{1}\cdots i_{n+2}})<1$ for all $i_{1}\cdots i_{n+2}$,
then the image of $f$ is contained in a projective
linear subspace over $\mathcal{P}_{c}^{1}$ of dimension $\leq [n/p]$.
\end{theorem}

We have introduced the holomorphic mapping \eqref{bfa} only for the purpose of stating the relevant growth condition for the coordinate functions $a_{ji}$ of $\mathbf{a}_0,\ldots,\mathbf{a}_{n+p}$ in a condensed form. Alternatively this assumption could be replaced with a stronger but simpler condition that each of the coordinate functions $a_{ji}$ satisfy $\varsigma(a_{ji})<1$.
Note that in either case we do not need every element of $\mathbf{a}_{j}$ to be of growth $o(T_{g}(r))$, what is needed here is just that the hyper-order of the holomorphic mapping \eqref{bfa} is strictly less than $1$.

Before going into the proof, we demonstrate the sharpness of Theorem~\ref{thm1.312} by using the following example.

\begin{example}
Since $g(z):=\pi/\Gamma(1-z)=(\sin \pi z) \Gamma(z)$ is an entire
function with only simple zeros on the set of positive
integers, it follows that $g^{-1}(\{0\})=\mathbb{Z}_{> 0}$ is
forward invariant under the shift $\tau(z)=z+1$. On the other
hand, the entire function $h(z):=(\sin \pi z)/\Gamma(z)$ has
simple zeros on $\mathbb{Z}_{> 0}$ and double zeros on the set of
non-positive integers $\mathbb{Z}_{\leq 0}$. Despite of this jump in the multiplicities at the origin, the set of the zeros
of $h(z)$ is still forward invariant with respect to $\tau(z)$ in
our definition. We also note that the gamma function $\Gamma(z)$ is a meromorphic function of order $1$ and maximal type in the plane, in fact, $$T(r,\Gamma)=(1+o(1))\frac{r}{\pi}\log r$$ by, e.g., \cite[Proposition 7.3.6]{cherryy:01}, while $$T(r, \sin \pi z)=2r +O(1)=o(T(r, \Gamma) ), $$ see, e.g., \cite[p.~27]{cherryy:01}.
Further, $\sin \pi z \in \mathcal{P}_1^1$ but $\Gamma \not\in\mathcal{P}_1^1$.

Let us consider the holomorphic curve
$$
f:=\left[\frac{1}{\Gamma(z)}: \frac{1}{\Gamma(z)}:
\frac{1}{\Gamma(z+1/2)}\right]
=\left[1:1:\frac{\Gamma(z)}{\Gamma(z+1/2)}\right]:\mathbb{C}\to\mathbb{P}^2
$$
which has its image in a subset of $\mathbb{P}^2$ of
dimention~$1$. Take the four moving hyperplanes $H_j(z)$ over
$\mathcal{P}^1_c$ with $c=1$, each of which is given respectively by
the vectors
$$
\bigl(\sin \pi z, 0, 0\bigr)\,, \ \bigl(0, \sin \pi (z+1),
0\bigr)\,, \ \bigl(0, 0,\sin \pi (z+1/2)\bigr)\,,
$$
and
$$\bigl(\sin \pi z, \sin \pi(z+1), \sin \pi(z+1/2)\bigr)$$
in $(\mathcal{P}^1_1)^3$ in general position. Now it is easy to see that each of these
hyperplanes has a forward invariant preimage under $f$. For
example, $f^{-1}(\{H_1\})$ coincides with the zeros of the above
entire function $h(z)$. This shows that Theorem~\ref{thm1.312} is sharp in the case where $n=p=2$.

Similarly, when $n=3$ and $p=2, 3$, the bound~$[n/p]=1$ is
attained by the six hyperplanes given by following vectors
in~$(\mathcal{P}_1^1)^4$ in general position with the primitive fourth root of
unity~$\omega$:
$$
(\sin \pi z) \bigl(1, 0, 0, 0\bigr)\,, \ (\sin \pi z) \bigl(0, 1, 0,
0\bigr)\,, \ (\sin \pi z) \bigl(1, \omega, \omega^2, \omega^3
\bigr)\,,
$$
$$
(\cos \pi z)\bigl(0, 0, 1, 0\bigr)\,,  \ (\cos \pi z) \bigl(0, 0, 0,
1\bigr)\,, \ (\cos \pi z) \bigl(1, 1, 1, 1\bigr)\,
$$
and the curve $f:\mathbb{C}\to\mathbb{P}^3$ is given by
\begin{eqnarray*}
f&:=&\left[\frac{1}{\Gamma(z)}: -\frac{1}{\Gamma(z)}:
\frac{\omega}{\Gamma\left(z+\frac{1}{2}\right)}:-\frac{1}{\Gamma\left(z+\frac{1}{2}\right)}\right]\\
&=&\left[1:-1:\omega
\frac{\Gamma(z)}{\Gamma\left(z+\frac{1}{2}\right)}: -
\frac{\Gamma(z)}{\Gamma\left(z+\frac{1}{2}\right)}\right].
\end{eqnarray*}
This $f$ is linearly degenerate in the sense that
$$
f(\mathbb{C})=\bigl\{[z_1:z_2:z_3:z_4]\in\mathbb{P}^3~|~z_1+z_2=0,
\ z_3+ \omega z_4=0 \bigr\} \simeq \mathbb{P}^1\,.
$$
\end{example}

A counter-example is also given to show the best-possibility of the restriction of hyper-order $<1$.

\begin{example}
Consider the holomorphic curve $f(z):=[1 : \exp e^{2\pi i z}]: \mathbb{C}  \to \mathbb{P}^1$, and three two-dimensional constant vectors $(1, 0)$, $(0, -1)$, $(1, -1) $ associating to three hyperplanes of $\mathbb{P}^1$ in general position. It is easy to see that the roots of the linear equation $$ \langle (1, \exp e^{2\pi i z}), (1, -1)\rangle = 1- \exp e^{2\pi i z} = 0 $$ are forward invariant with respect to $\tau(z)=z+1$, since they are of the form $$ z = \frac{1}{2\pi i}\log(2 m \pi) \pm \frac{1}{4} + k $$ for $m\in\mathbb{Z}_{>0}$ and $k\in\mathbb{Z}$.
(For $\tau(z) = \frac{1}{2\pi i}\log(2 m \pi) \pm \frac{1}{4} + (k + 1).
$)
On the other hand, $[n/p]=[1/2]=0$, but $f(z)$ satisfies $f\not\in \mathcal{P}_{1}^{1}$.
\end{example}

In order to prove Theorem~\ref{thm1.312}, we need to introduce the following lemma, which is a generalization of \cite[Theorem 3.1]{halburdkt:14}. Since the proof is a simple modification of the proof of \cite[Theorem 3.1]{halburdkt:14}, we omit the details.


\begin{lemma}\label{thm1.3121}
Let $c\in \mathbb{C}$, and let $g=[g_{0}:\cdots:g_{n}]$ be a holomorphic
curve such that $\varsigma^{*}(g)<1$ and such that preimages of all zeros
of $g_{0},\, \ldots,\, g_{n}$ are forward invariant with respect to the translation
 $\tau(z)=z+c$. Let
 $$S_{1}\cup\cdots \cup S_{l}$$
 be the partition of $\{0,\, \ldots,\, n\}$ formed in such a way that
 $i$ and $j$ are in the same class $S_{k}$ if and only if $g_{i}/g_{j}\in\mathcal{P}_{c}^{1}$.
 If
 \begin{eqnarray}\label{thm1.3121eq1}
   \sum_{i=0}^{n}c_{i}g_{i}\equiv 0,
 \end{eqnarray}
 where $c_{i}\in\mathcal{P}_{c}^{1}$, then
 \begin{eqnarray*}
   \sum_{i\in S_{k}}c_{i}g_{i}\equiv 0
 \end{eqnarray*}
 for all $k\in \{1, \ldots, l\}$.
 \end{lemma}

\begin{proof}[Proof of Theorem~\ref{thm1.312}]
Let $x=[x_{0}:\cdots:x_{n}]$, and let $L_{H_j}(z)\, (j=1, \ldots, n+p)$ be the linear forms defining the hyperplanes $H_j(z)=0$ as in \eqref{thm1.31eq11}. Since
by assumption any $n+1$ of the hyperplanes $H_{j},\, j=1, \ldots, n+p$,
are linearly independent over $\mathcal{P}_{c}^{1}$, it follows that
any $n+2$ of the forms $L_{H_j}(z)$ satisfy a linear relation with coefficients none of which vanishes identically in
$\mathcal{P}_{c}^{1}$. By writing $\tau(z)=z+c$, it follows by assumption that the functions $h_{j}=L_{H_{j}}(f,a)=a_{j0}f_{0}+\cdots+a_{jn}f_{n}$ satisfy
    \begin{eqnarray*}
    \{\tau(h_{j}^{-1}(\{0\}))\}\subset \{h_{j}^{-1}(\{0\})\}
    \end{eqnarray*}
for all $j=1,\ldots, n+p$, where $\{\cdot\}$ denotes a multiset which takes into
account the multiplicities of its elements.

The set of indexes $\{1, \ldots, n+p\}$ may be split into disjoint equivalence classes
$S_{k}$ by saying that $i\sim j$ if $h_{i}=\alpha h_{j}$ for some $\alpha \in \mathcal{P}_{c}^{1}\backslash \{0\}$.
Therefore
    \begin{eqnarray*}
    \{1, \ldots, n+p\}=\bigcup_{j=1}^{N}S_{j}
    \end{eqnarray*}
for some $N\in \{1, \ldots, n+p\}$.

Suppose that the complement of $S_{k}$ has at least $n+1$ elements for some
$k\in \{1, \ldots, N\}$. Choose an element $s_{0}\in S_{k}$, and denote $U=\{1,\ldots, n+p\}\backslash S_{k}\cup \{s_{0}\}$.
Since the set $U$ contains at least $n+2$ elements, there exists a subset $U_{0}\subset U$
 such that $U_{0}\cap S_{k}=\{s_{0}\}$ and $\card (U_{0})=n+2$. Therefore, there exists
 $c_{j}\in \mathcal{P}_{c}^{1}\backslash \{0\}$ such that
\begin{eqnarray}\label{mnbvcxz}
  \sum_{j\in U_{0}} c_{j}h_{j}\equiv 0.
\end{eqnarray}

Denote $h=[h_{i_{1}}: \cdots : h_{i_{n+2}}]$. In order to apply Lemma~\ref{thm1.3121} to deduce a contradiction, we need to prove
$\varsigma^{*}(h)<1$. To this end, let
    $$u(z)=\sup_{k\in \{0,\ldots,n\}}\log |f_{k}(z)|$$
and
   $$v_{i_{j}}(z)=\sup_{k\in \{0,\ldots,n\}}\log |a_{i_{j}k}(z)| \qquad (j=1,\ldots,n+2).$$
Then we have
\begin{eqnarray}\label{sec3eqa}
  \log |h_{i_{j}}|&=&\log |a_{i_{j}0}f_{0}+\cdots+a_{i_{j}n}f_{n}| \nonumber\\
  &\leq& \log \left(|a_{i_{j}0}f_{0}|+\cdots+|a_{i_{j}n}f_{n}| \right)\nonumber\\
  &\leq& \log (n+1)\cdot e^{u}\cdot e^{v_{i_{j}}} \nonumber\\
   &\leq& u+v_{i_{j}}+O(1)\nonumber\\
   &=& \sup_{k\in \{0,\ldots,n\}}\log |f_{k}|+ \sup_{k\in \{0,\ldots,n\}}\log |a_{i_{j}k}|+O(1)
\end{eqnarray}
for any $z$ satisfying $|h_{i_{j}}(z)|\neq 0$ and $\sup_{k\in \{0,\ldots,n\}}|a_{i_{j}k}(z)| \neq 0$.
Thus we have
\begin{eqnarray*}
  \sup_{j\in \{1,\ldots, n+2\}}\log |h_{i_{j}}|\leq \sup_{k\in \{0,\ldots,n\}}\log |f_{k}|
  +\sup_{j\in \{1,\ldots, n+2\}}\sup_{k\in \{0,\ldots,n\}}\log |a_{i_{j}k}|+O(1)
\end{eqnarray*}
for any $z$ satisfying
\begin{equation}\label{notzero}
\begin{split}
& \sup_{j\in \{1,\ldots, n+2\}}|h_{i_{j}}(z)|\neq 0,\\
& \sup_{j\in \{1,\ldots, n+2\}} \sup_{k\in \{0,\ldots,n\}}|a_{i_{j}k}(z)| \neq 0.
\end{split}
\end{equation}
This gives that
\begin{eqnarray}\label{sec3eqb}
 &&  \int_{0}^{2\pi}\sup_{j\in \{1,\ldots, n+2\}} \log |h_{i_{j}}(re^{i\theta})|\frac{d\theta}{2\pi}
  \leq\int_{0}^{2\pi}\sup_{k\in \{0,\ldots,n\}}\log |f_{k}(re^{i\theta})|\frac{d\theta}{2\pi}\nonumber\\
  &&+
  \int_{0}^{2\pi}\sup_{j\in \{1,\ldots, n+2\}}\sup_{k\in \{0,\ldots,n\}}\log |a_{i_{j}k}(re^{i\theta})|\frac{d\theta}{2\pi}+O(1)
\end{eqnarray}
for those positive $r$ for which the functions in \eqref{notzero} have no zeros on $|z|=r$.

Suppose that $\sup_{j\in \{1,\ldots, n+2\}}|h_{i_{j}}|$ has infinitely many zeros on the circle $\{z:|z|=r\}$
(where $r>0$). Then from Bolzano--Weierstrass theorem, there exists a  convergent subsequence $z_{t}\to z_{0}$
 as $t \to \infty$ satisfying
 $$
 \sup_{j\in \{1,\ldots, n+2\}}|h_{i_{j}}(z_{t})|=\sup_{j\in \{1,\ldots, n+2\}}|h_{i_{j}}(z_{0})|=0.
 $$
Thus we have $|h_{i_{j}}(z_{t})|=|h_{i_{j}}(z_{0})|=0\,
 (j=1,\ldots, n+2)$, i.e., $h_{i_{j}}(z_{t})=h_{i_{j}}(z_{0})=0\, (j=1,\ldots, n+2)$. Since $h_{i_{j}}\,(j=1,\ldots, n+2)$
 are all entire functions, it follows from the identity theorem of holomorphic functions, that $h_{i_{j}} \equiv 0 \,(j=1,\ldots, n+2)$, a contradiction.
Similarly it follows that also $\sup_{j\in \{1,\ldots, n+2\}} \sup_{k\in \{0,\ldots,n\}}|a_{i_{j}k}|$  can have at most finitely many zeros on the circle $\{z:|z|=r\}$.


If either one of the functions  $\sup_{j\in \{1,\ldots, n+2\}}
\sup_{k\in \{0,\ldots,n\}}|a_{i_{j}k}|$ and $\sup_{j\in \{1,\ldots, n+2\}}|h_{i_{j}}|$  have a finite number of zeros on the circle $\{z:|z|=r\}$
(where $r>0$), we refer to \cite{gundersenh:04} for the method of proof on how to deal with this case. Here in convenience for the readers,
we give the details of the proof for \eqref{sec3eqb}, following \cite{gundersenh:04}.  For those $r$ where there are zeros on the circle of radius $r$, we modify the path of integration slightly in order to avoid having zeros on the path. This is done by integrating the three integrands in \eqref{sec3eqb} around a curve $\gamma=\gamma(r,\delta)$ consisting of arcs of $|z|=r$ and small
``recesses'' of sufficiently small radius $\delta$ about each zero of $\sup_{j\in \{1,\ldots, n+2\}}|h_{i_{j}}|$ and
 $\sup_{j\in \{1,\ldots, n+2\}}
\sup_{k\in \{0,\ldots,n\}}|a_{i_{j}k}|$ on
$|z|=r$, such that these functions have no zeros on the new path of integration for any $\delta>0$. In this case, \eqref{sec3eqb} holds when the path of integration is replaced by
$\gamma$. Letting $\delta \to 0$, it follows that on each small recess
the integrands on both sides of the inequality \eqref{sec3eqb} are of the form $O(-\log \delta)$,
and the length of the recess is of the form $O(\delta)$. This implies that the corresponding integrals
around each recess tend to zero as $\delta\to0$. Since the curve $\gamma$ approaches the
circle $|z|=r$ as $\delta \to 0$, it follows that \eqref{sec3eqb} holds on $|z|=r$.

We have shown that \eqref{sec3eqb} holds for all positive $r$.
We set
$$T^*_{h}(r)=\int_{0}^{2\pi}\sup_{j\in \{1,\ldots, n+2\}} \log |h_{i_{j}}(re^{i\theta})|\frac{d\theta}{2\pi},$$
$$T^*_{f}(r)=\int_{0}^{2\pi}\sup_{k\in \{0,\ldots,n\}}\log |f_{k}(re^{i\theta})|\frac{d\theta}{2\pi},$$
and
$$T^*_{\mathbf{a}_{i_{1}\cdots i_{n+2}}}(r)=\int_{0}^{2\pi}\sup_{j\in \{1,\ldots, n+2\}}\sup_{k\in \{0,\ldots,n\}}\log |a_{i_{j}k}(re^{i\theta})|\frac{d\theta}{2\pi}.$$
Then from \eqref{sec3eqb} we have
    \begin{equation}\label{thm1.3121eq1212}
    T^*_{h}(r)\leq T^*_{f}(r)+T^*_{\mathbf{a}_{i_{1}\cdots i_{n+2}}}(r)+O(1).
    \end{equation}
Since by assumption $\varsigma(f)<1$ and $\varsigma^*(\mathbf{a}_{i_{1}\cdots i_{n+2}})<1$, it follows by \eqref{thm1.3121eq1212} that
    \begin{equation}\label{thm1.3121eq1212a}
    \varsigma^{*}(h)<1
    \end{equation}
and we can hence apply Lemma~\ref{thm1.3121}.

By using Lemma~\ref{thm1.3121}, from \eqref{mnbvcxz} we get that
\begin{eqnarray*}
c_{s_{0}}h_{s_{0}} \equiv 0,
\end{eqnarray*}
which is a contradiction. So we have that the set $\{1, \ldots, n+p\}\backslash S_{k}$
has at most $n$ elements. Hence $S_{k}$ has at least $p$ elements
for all $k=1,\ldots, N$, and it follows that $N\leq (n+p)/p$.

Let $V$ be any subset of $\{1, \ldots, n+p\}$ with exactly $n+1$ elements.
Then the forms $L_{H_j}$, $j\in V$, are linearly independent. By denoting
$V_{k}=V\cap S_{k}$ it follows that
\begin{eqnarray*}
  V=\bigcup _{k=1}^{N}V_{k}.
\end{eqnarray*}
Since each set $V_{k}$ gives raise to $\card (V_{k})-1$ equations over the
field $\mathcal{P}_{c}^{1}$, it follows that
we have at least
\begin{eqnarray*}
  \sum_{k=1}^{N}\left(\card (V_{k})-1 \right)=n+1-N \geq n+1-\frac{n+p}{p}=n-\frac{n}{p}
\end{eqnarray*}
linearly independent relations over the field $\mathcal{P}_{c}^{1}$. Therefore the
image of $f$ is contained in a linear subspace over $\mathcal{P}_{c}^{1}$ of
dimension $\leq [n/p]$, as desired.
\end{proof}

\section{Difference analogues of truncated second main theorem}\label{deficiencysec}

In this section we introduce two alternative difference analogues of the truncated second main theorem, and give corresponding difference deficiency relations. We start with a definition of the difference counterpart of the concept of truncation.

\begin{definition}
Let $n\in \N$, $c\in\C\setminus\{0\}$ and $a\in \P$.
An $a$-point $z_0$ of a meromorphic function $h(z)$ is said to be $n$-{\bf successive} and
$c$-{\bf separated}, if the $n$ entire functions $h(z+\nu c)$ $(\nu=1,\ldots, n)$ take
the value~$a$ at $z=z_0$ with multiplicity not less than that of $h(z)$ there.
All the other $a$-points of $h(z)$ are called $n$-{\bf aperiodic} of {\bf pace}~$c$.
By $\tilde{N}_g^{[n,c]}(r,L_{H})$ we denote the counting function of $n$-{\bf aperiodic} zeros
of the function $L_H(g,\mathbf{a})=\langle g(z), \mathbf{a}(z)\rangle$ of pace~$c$.
\end{definition}

Note that $\tilde{N}_g^{[n,c]}(r,L_H)\equiv 0$ when all the zeros of $L_H(g,\mathbf{a})$
with taking their multiplicities into account are located periodically with period~$c$.
This is also the case when the hyperplane~$H$ is forward invariant by~$g$
with respect to the translation $\tau_c(z)=z+c$, i.e.
$\tau_c\bigl(g^{-1}(\{H\})\bigr) \subset g^{-1}(\{H\})$ holds.
In fact, it follows by definition that any zero with a forward invariant preimage of the function $L_H(g,\mathbf{a})$ must
be $n$-successive and $c$-separated, since
$$
g^{-1}(\{H\}) \subset \tau_{-c}\bigl(g^{-1}(\{H\})\bigr) \subset \tau_{-(n-1)c}\bigl(g^{-1}(\{H\})\bigr).
$$
In addition, we denote
\begin{eqnarray*}
  N_{g}(r,L_H)=N\left(r,\frac{1}{L_H(g,\mathbf{a})} \right)=N\left(r,\frac{1}{\langle g(z), \mathbf{a}(z)\rangle} \right)
\end{eqnarray*}
and
\begin{eqnarray*}
  N_{C}(r,0)=N\left(r,\frac{1}{C(g_{0},\ldots,g_{n})}\right).
\end{eqnarray*}

We give the following short notation to be used through the remainder of this paper.
Let $g(z)$ be a meromorphic function, and let $c\in \mathbb{C}$, we set
\begin{eqnarray*}
g(z)\equiv g, \, g(z+c)\equiv \overline{g}, \, g(z+2c)\equiv \overline{\overline{g}}\,  \textrm{and}\,  g(z+nc)\equiv \overline{g}^{[n]}
\end{eqnarray*}
to suppress the $z$-dependence of $g(z)$. The Casorati determinant of $g_{0},\, \ldots,\, g_{n}$ is then defined by
\begin{displaymath}
C(g_{0},\, \ldots,\, g_{n}) =
\left| \begin{array}{cccc}
g_{0} & g_{1} & \ldots & g_{n} \\
\overline{g}_{0} & \overline{g}_{1} & \ldots  &  \overline{g}_{n} \\
\vdots & \vdots & \ddots & \vdots   \\
\overline{g}_{0}^{[n]} & \overline{g}_{1}^{[n]} &  \ldots & \overline{g}_{n}^{[n]}
\end{array} \right|.
\end{displaymath}

With these definitions in hand we can show the following auxiliary result.

\begin{lemma}\label{movinglemma}
Let $g$ be a holomorphic curve of $\C$ into $\P^n(\C)$, let $n\in\N$ and $q\in\N$ be such that $q\geq n$, and let
    $$
    \mathbf{a}_j(z) = (a_{j0},\ldots,a_{jn}), \qquad j\in\{0,\ldots,q\},
    $$
where $a_{jk}(z)$ are $c$-periodic entire functions satisfying $T(r,a_{jk})=o(T_{g}(r))$ for all $j,k\in\{0,\ldots,q\}$.
If the moving hyperplanes
    \begin{equation}\label{Hj}
    H_j(z) = \left\{[x_0:\cdots:x_n] : L_{H_j}(x,\mathbf{a}_j(z))=0\right\}, \qquad j\in\{0,\ldots,q\},
    \end{equation}
are located in general position, then
\begin{equation}\label{eq}
\sum_{j=0}^q N_g(r,L_{H_{j}}) -N_{C}(r, 0) \leq \sum_{j=0}^q \tilde{N}_g^{[n,c]}(r, L_{H_{j}})+o(T_{g}(r)).
\end{equation}
\end{lemma}

\begin{proof}
By Theorem~\ref{lem2.1} the counting function $N^*(r)$ for those points where more than $n$ functions $ L_{H_j}$ vanish simultaneously is of the growth
    \begin{equation}\label{Nstar2}
    N^*(r) = o(T_g(r)).
    \end{equation}
The contribution to \eqref{eq} from such points can therefore be incorporated in the error term.

Suppose now that $z_0$ is an $n$-successive $c$-separated zero of $L_{H_j}$ for some $j\in \{0,\ldots,q\}$. By \eqref{Nstar2}, and Theorem~\ref{lem2.1},  we may assume that there are at most $n$ indexes within $\{0,\ldots,q\}$ such that $L_{H_j}(z_0)=0$. Therefore, by reordering the indexes if necessary we may assume that $L_{H_j}(z_0)\not=0$ for all $j\in\{n,\ldots,q\}$, and thus there is no contribution to the counting functions $N_g(r, L_{H_n}),\ldots,N_g(r, L_{H_q})$ from the point $z_0$.

Now, there are integers $m_{j}(\geq 0)$ and holomorphic functions $h_{jk}(z)$ in a neighborhood~$U$ of~$z_0$
such that
\begin{equation}\label{mj}
L_{H_j}(z+k c)=\langle g(z+k c), \mathbf{a}_{j}(z+k c) \rangle = (z-z_0)^{m_{j}} h_{jk}(z) \quad \text{for} \
0\leq j,k \leq n.
\end{equation}
Here, for convenience, we set $m_{j}= 0$ whenever $\langle g(z_{0}), \mathbf{a}_j(z_{0})\rangle \neq 0$. Since
$$
L_{H_j}(z)=\langle g(z), \mathbf{a}_{j}(z) \rangle=\sum_{k=0}^n a_{jk}(z)g_k(z),
$$
where $a_{jk}$ are $c$-periodic entire functions satisfying $T(r,a_{jk})=o(T_{g}(r))$, it follows that
\begin{eqnarray*}
& & \hspace{-15mm} \left(
\begin{array}{cccc}
L_{H_0} 	& 	L_{H_1} 	& \cdots 	& L_{H_n} \\[1ex]
\overline{L}_{H_0}		& \overline{L}_{H_1} 	& \cdots 	& \overline{L}_{H_n} \\[1ex]
\vdots  	& 	\vdots 	& 		& 	\vdots 	\\[1ex]
\overline{L}_{H_0}^{[n]} 	& \overline{L}_{H_1}^{[n]} 	& 	\cdots 	& \overline{L}_{H_n}^{[n]}
\end{array}
\right)
=
\left(
\begin{array}{cccc}
g_0 	& 	g_1 	& \cdots 	& g_n 		\\[1ex]
\overline{g}_0		& \overline{g}_1 	& \cdots 	& 	\overline{g}_n 	\\[1ex]
\vdots  	& 	\vdots 	& 		& 	\vdots 	\\[1ex]
\overline{g}_0^{[n]} 	& \overline{g}_1^{[n]} 	& 	\cdots 	& \overline{g}_n^{[n]}	
\end{array}
\right) \times A,
\end{eqnarray*}
where
    \begin{equation*}
    A=\left(
\begin{array}{cccc}
a_{00} 	& 	a_{10} 	& \cdots 	& 	a_{n0} 	\\[1ex]
a_{01}		& 	a_{11} 	& \cdots 	& 	a_{n1} 	\\[1ex]
\vdots  	& 	\vdots 	& 		& 	\vdots 	\\[1ex]
a_{0n} 	& 	a_{1n} 	& \cdots 	& 	a_{nn}	
\end{array}
\right) \,.
    \end{equation*}
Since the hyperplanes \eqref{Hj} are in general position, we may invert $A$ and obtain
$$
C(g_0, \ldots , g_n) = C\bigl(L_{H_0}, \ldots, L_{H_n}\bigr) \det(A^{-1})\,,
$$
where $T(r,\det(A^{-1}))=o(T_{g}(r))$. The cases where $z_0$ is a zero of $\det(A^{-1})$ can therefore be incorporated in the error term $o(T_{g}(r))$. Hence, assuming that  $\det(A^{-1})$ is non-zero at $z_0$, we have by \eqref{mj},
$$
C\bigl(g_0, \ldots, g_n\bigr)
=\prod_{j=0}^{n}(z-z_0)^{m_j}h(z)\,,
$$
where $h(z)$ is a holomorphic function defined on~$U$. Thus $C(g_0, \ldots, g_n)$ vanishes at~$z$ with order at least $\sum_{j=1}^{q}m_j$. This, by going through all points $z_0\in\C$, together with definitions of $N_g(r, L_{H_j})$, $N_C(r,0)$ and $\tilde{N}_g^{[n,c]}(r,L_{H_j})$ implies the assertion.
\end{proof}

The following difference analogue of truncated second main theorem is an application of Lemma~\ref{movinglemma} and Theorem~\ref{thm1.2}.

\begin{theorem}[Difference analogue of Cartan's Second Main Theorem]\label{dCartantr}
Let $n\geq 1$, and let $g=[g_{0}: \ldots: g_{n}]$ be a holomorphic curve of $\C$ into $\P^n(\C)$ with $\varsigma := \varsigma(g)<1$, where $g_{0},\, \ldots, g_{n}$ are 
linearly independent over $\mathcal{P}_{c}^{1}$. Let
    $$
    \mathbf{a}_j(z) = (a_{j0},\ldots,a_{jn}), \qquad j\in\{0,\ldots,q\},
    $$
where $a_{jk}(z)$ are $c-$periodic entire functions satisfying $T(r,a_{jk})=o(T_{g}(r))$ for all $j,k\in\{0,\ldots,q\}$.
If the moving hyperplanes
    \begin{equation*}
    H_j(z) = \left\{[x_0:\cdots:x_n] : L_{H_j}(x,\mathbf{a}_j(z))=0\right\}, \qquad j\in\{0,\ldots,q\},
    \end{equation*}
are located in general position, then
\begin{equation*}
(q-n)T_g(r) \leq  \sum_{j=0}^q \tilde{N}_g^{[n,c]}(r, L_{H_j}) +o(T_{g}(r))
\end{equation*}
for all $r$ outside of a set $E$ with finite logarithmic measure.
\end{theorem}

From Theorem~\ref{dCartantr} we can obtain a difference analogue of the truncated deficiency relation for holomorphic curves.

\begin{corollary}
Under the assumptions of Theorem~\ref{dCartantr}, we have
\begin{eqnarray*}
\sum_{j=0}^{q}\delta^{[n,c]}_{g}(0, L_{H_j})\leq n+1,
\end{eqnarray*}
where
\begin{eqnarray*}
\delta^{[n,c]}_{g}(0, L_{H_j})=1-\limsup_{r\to \infty}\frac{\tilde{N}_g^{[n,c]}(r,  L_{H_j})}{T_{g}(r)}.
\end{eqnarray*}
\end{corollary}

%

Instead of $n$-successive points, we can consider points with different
separation properties. For instance, we say that $a$ is a \textbf{\emph{derivative-like}} paired value of $f$
with the separation $c$ if the following property holds for all except at most finitely many $a$-points of $f$:
whenever $f(z)=a$ with the multiplicity $m$, then also $f(z+c)=a$ with the multiplicity $\max \{m-1, 0\}$.

Before introducing Theorem~\ref{thm2.1}, we give the following definition of the usual truncated counting function first, please refer, for instance,
to \cite{gundersenh:04} for details.

\begin{definition} \label{def1.1}
For a meromorphic function $f$ satisfying $f\not\equiv0$ and a
positive integer $j$, let $n_{j}(r,0,f)$ denote the number of zeros of
$f$ in $\{z: |z|\leq r\}$, counted in the following manner: a zero of
$f$ of multiplicity $m$ is counted exactly $k$ times where $k=\min\{m,j\}$. Then let $N_{j}(r,0,f)$ denote
the corresponding integrated counting function; that is,
\begin{eqnarray*}
  N_{j}(r,0,f)=n_{j}(0,0,f)\log r+
  \int_{0}^{r}\frac{n_{j}(t,0,f)-n_{j}(0,0,f)}{t}dt.
\end{eqnarray*}
\end{definition}

With this definition we may state the second difference analogue of the truncated second main theorem.

\begin{theorem} \label{thm2.1}
Assume that the hypotheses of Theorem \ref{thm1.2} hold, and $0$ is a derivative-like
paired value of $f_{i}$ with the separation $c$ for all $i\in\{0,\ldots,q\}$. Then we have
\begin{equation*}
N(r,0,L)\leq \sum _{j=0}^{q}N_{n}(r,0,f_{j})+O(1),
\end{equation*}
and this gives
\begin{equation*}
(q-n)T_g(r)\leq \sum _{j=0}^{q}N_{n}(r,0,f_{j})-N(r,L)+o(T_{g}(r)),
\end{equation*}
where $r$ approaches infinity outside of an exceptional set of finite logarithmic measure.
\end{theorem}

For the proof of Theorem \ref{thm1.2} and Theorem \ref{thm2.1}, the following lemma is needed.

\begin{lemma}[\cite{halburdkt:14}] \label{lem2.6}
If the holomorphic curve $g=[g_{0}:\cdots:g_{n}]$
satisfies $\varsigma(g)<1$ and if $c\in \mathbb{C}$, then $C(g_{0},\ldots, g_{n})\equiv 0$ if and only if the entire
functions $g_{0}, \ldots, g_{n}$ are linearly dependent over the
field $\mathcal{P}_{c}^{1}$.
\end{lemma}

\begin{proof}[Proof of Theorem~\ref{thm2.1}]
Suppose that $a_{0}, \ldots, a_{q-n-1}$ are any $q-n$ distinct integers
in the set $\{0, 1, \ldots, q\}$, and let $b_{0}, b_{1}, \ldots, b_{n}$
denote the remaining integers in the set $\{0,1, \ldots, q\}$.

From the assumptions of Theorem~\ref{thm2.1} (the same as the assumptions of Theorem~\ref{thm1.2}), we have that
\begin{eqnarray}\label{sec3aaa}
\left( \begin{array}{ccc}
f_{b_{0}} & \cdots & f_{b_{n}} \\
\overline{f}_{b_{0}} & \cdots  & \overline{f}_{b_{n}} \\
\vdots & \ddots & \vdots \\
\overline{f}_{b_{0}}^{[n]}& \cdots & \overline{f}_{b_{n}}^{[n]}
\end{array} \right)
=
\left( \begin{array}{ccc}
g_{0} & \cdots & g_{n} \\
\overline{g}_{0} & \cdots  & \overline{g}_{n} \\
\vdots & \ddots & \vdots \\
\overline{g}_{0}^{[n]}& \cdots & \overline{g}_{n}^{[n]}
\end{array} \right)
\cdot
\left( \begin{array}{ccc}
a_{0b_{0}} & \cdots & a_{0b_{n}} \\
a_{1b_{0}} & \cdots  & a_{1b_{n}} \\
\vdots & \ddots & \vdots \\
a_{nb_{0}}& \cdots & a_{nb_{n}}
\end{array} \right),
\end{eqnarray}
where

\begin{eqnarray}\label{sec3aabbbb}
\left| \begin{array}{ccc}
a_{0b_{0}} & \cdots & a_{0b_{n}} \\
a_{1b_{0}} & \cdots  & a_{1b_{n}} \\
\vdots & \ddots & \vdots \\
a_{nb_{0}}& \cdots & a_{nb_{n}}
\end{array} \right|=
\left| \begin{array}{ccc}
a_{0b_{0}} & \cdots &  a_{nb_{0}}\\
a_{0b_{1}} & \cdots  & a_{nb_{1}} \\
\vdots & \ddots & \vdots \\
a_{0b_{n}}& \cdots & a_{nb_{n}}
\end{array} \right|=d_{b_{0}\cdots b_{n}}(z).
\end{eqnarray}
Lemma \ref{lem2.6} yields $C(g_{0},\, \ldots,\, g_{n})\not\equiv 0$. Since $f_{b_{0}},\, \ldots,\, f_{b_{n}}$ are linearly independent
over $\mathcal{P}_{c}^{1}$,  it follows that the determinant \eqref{sec3aabbbb} of the coefficient matrix of \eqref{sec3aaa} satisfies
\begin{eqnarray*}\label{sec3aab}
d_{b_{0}\cdots b_{n}}(z)\not\equiv 0,
\end{eqnarray*}
and so also $C(f_{b_{0}},\, \ldots,\, f_{b_{n}})\not\equiv 0$ by \eqref{sec3aaa}. Since
\begin{eqnarray}\label{sec3aac}
\left(
\begin{array}{c}
f_{b_{0}}(z)\\
f_{b_{1}}(z)\\
\vdots \\
f_{b_{n}}(z)\\
\end{array}
\right)
=
\left( \begin{array}{ccc}
a_{0b_{0}} & \cdots & a_{nb_{0}} \\
a_{0b_{1}} & \cdots  & a_{nb_{1}} \\
\vdots & \ddots & \vdots \\
a_{0b_{n}} & \cdots & a_{nb_{n}}
\end{array} \right)
\cdot
\left(
\begin{array}{c}
g_{0}(z)\\
g_{1}(z)\\
\vdots \\
g_{n}(z)\\
\end{array}
\right),
\end{eqnarray}
by using Cramer's rule, we get that
\begin{eqnarray*}\label{sec3eaee}
g_{i} &=&\frac{\left|
\begin{array}{ccccccc}
a_{0b_{0}} & \cdots  &  a_{i-1 b_{0}}  & f_{b_{0}} &  a_{i+1 b_{0}}  &  \cdots  &  a_{nb_{0}} \\
a_{0b_{1}} & \cdots  &  a_{i-1 b_{1}} & f_{b_{1}} &  a_{i+1 b_{1}} &\cdots  & a_{nb_{1}} \\
\vdots &  \cdots &  \vdots  & \vdots &\vdots &  \cdots  &\vdots \\
a_{0b_{n}} & \cdots  &  a_{i-1 b_{n}} &  f_{b_{n}}& a_{i+1 b_{n}}&\cdots  & a_{nb_{n}}
\end{array} \right|}{d_{b_{0}\cdots b_{n}}(z)}\\
&=&
(-1)^{1+i+1}\frac{\left|
\begin{array}{cccccc}
a_{0b_{1}} & \cdots  &  a_{i-1 b_{1}}  &  a_{i+1 b_{1}} &\cdots  & a_{nb_{1}} \\
\vdots &  \cdots &  \vdots  &\vdots &  \cdots  &\vdots \\
a_{0b_{n}} & \cdots  &  a_{i-1 b_{n}} & a_{i+1 b_{n}}&\cdots  & a_{nb_{n}}
\end{array} \right|}{d_{b_{0}\cdots b_{n}}(z)}\cdot f_{b_{0}}\\
&+&(-1)^{2+i+1}\frac{
\left|
\begin{array}{cccccc}
a_{0b_{0}} & \cdots  &  a_{i-1 b_{0}}  &  a_{i+1 b_{0}}  &  \cdots  &  a_{nb_{0}} \\
a_{0b_{2}} & \cdots  &  a_{i-1 b_{2}}  &  a_{i+1 b_{2}}  &  \cdots  &  a_{nb_{2}}\\
\vdots &  \cdots &  \vdots  & \vdots  &  \cdots  &\vdots \\
a_{0b_{n}} & \cdots  &  a_{i-1 b_{n}} & a_{i+1 b_{n}}&\cdots  & a_{nb_{n}}
\end{array} \right|}{d_{b_{0}\cdots b_{n}}(z)}
\cdot f_{b_{1}} \\
&+&\cdots \\
&+&
(-1)^{n+1+i+1}
\frac{\left|
\begin{array}{cccccc}
a_{0b_{0}} & \cdots  &  a_{i-1 b_{0}}  &  a_{i+1 b_{0}}  &  \cdots  &  a_{nb_{0}} \\
a_{0b_{1}} & \cdots  &  a_{i-1 b_{1}}  &  a_{i+1 b_{1}}  &  \cdots  &  a_{nb_{1}}\\
\vdots &  \cdots &  \vdots  & \vdots  &  \cdots  &\vdots \\
a_{0b_{n-1}} & \cdots  &  a_{i-1 b_{n-1}} & a_{i+1 b_{n-1}}&\cdots  & a_{nb_{n-1}}
\end{array} \right|}{d_{b_{0}\cdots b_{n}}(z)}
\cdot f_{b_{n}}.
\end{eqnarray*}
We set
\begin{eqnarray}\label{sec3aad}
c^{*}_{b_{0}i}(z)&:=&
(-1)^{1+i+1}\frac{\left|
\begin{array}{cccccc}
a_{0b_{1}} & \cdots  &  a_{i-1 b_{1}}  &  a_{i+1 b_{1}} &\cdots  & a_{nb_{1}} \\
\vdots &  \cdots &  \vdots  &\vdots &  \cdots  &\vdots \\
a_{0b_{n}} & \cdots  &  a_{i-1 b_{n}} & a_{i+1 b_{n}}&\cdots  & a_{nb_{n}}
\end{array} \right|}{d_{b_{0}\cdots b_{n}}(z)}      \nonumber \\
c^{*}_{b_{1}i}(z) &:=&
(-1)^{2+i+1}\frac{
\left|
\begin{array}{cccccc}
a_{0b_{0}} & \cdots  &  a_{i-1 b_{0}}  &  a_{i+1 b_{0}}  &  \cdots  &  a_{nb_{0}} \\
a_{0b_{2}} & \cdots  &  a_{i-1 b_{2}}  &  a_{i+1 b_{2}}  &  \cdots  &  a_{nb_{2}}\\
\vdots &  \cdots &  \vdots  & \vdots  &  \cdots  &\vdots \\
a_{0b_{n}} & \cdots  &  a_{i-1 b_{n}} & a_{i+1 b_{n}}&\cdots  & a_{nb_{n}}
\end{array} \right|}{d_{b_{0}\cdots b_{n}}(z)}   \\
&\cdots&    \nonumber \\
c^{*}_{b_{n}i}(z) &:=&
(-1)^{n+1+i+1}
\frac{\left|
\begin{array}{cccccc}
a_{0b_{0}} & \cdots  &  a_{i-1 b_{0}}  &  a_{i+1 b_{0}}  &  \cdots  &  a_{nb_{0}} \\
a_{0b_{1}} & \cdots  &  a_{i-1 b_{1}}  &  a_{i+1 b_{1}}  &  \cdots  &  a_{nb_{1}}\\
\vdots &  \cdots &  \vdots  & \vdots  &  \cdots  &\vdots \\
a_{0b_{n-1}} & \cdots  &  a_{i-1 b_{n-1}} & a_{i+1 b_{n-1}}&\cdots  & a_{nb_{n-1}}
\end{array} \right|}{d_{b_{0}\cdots b_{n}}(z)} \nonumber
\end{eqnarray}
to obtain
\begin{eqnarray}\label{sec3aae}
\left(
\begin{array}{c}
g_{0}(z)\\
g_{1}(z)\\
\vdots \\
g_{n}(z)\\
\end{array}
\right)
&=&
\left( \begin{array}{cccc}
c^{*}_{b_{0}0}(z) & c^{*}_{b_{1}0}(z)&  \cdots & c^{*}_{b_{n}0}(z) \\
c^{*}_{b_{0}1}(z) & c^{*}_{b_{1}1}(z)& \cdots  & c^{*}_{b_{n}1}(z) \\
\vdots & \vdots & \ddots & \vdots \\
c^{*}_{b_{0}n}(z) &c^{*}_{b_{1}n}(z)& \cdots & c^{*}_{b_{n}n}(z)
\end{array} \right)
\cdot
\left(
\begin{array}{c}
f_{b_{0}}(z)\\
f_{b_{1}}(z)\\
\vdots \\
f_{b_{n}}(z)\\
\end{array}
\right),
\end{eqnarray}
where
$c^{*}_{b_{k}j}(k=0, \ldots,n; j= 0,\ldots, n)$ are $c$-periodic  meromorphic functions. Thus for all $i=1, \ldots,n$,
\begin{eqnarray}\label{sec3aag}
\left(
\begin{array}{c}
\overline{g}^{[i]}_{0}(z)\\
\overline{g}^{[i]}_{1}(z)\\
\vdots \\
\overline{g}^{[i]}_{n}(z)\\
\end{array}
\right)
&=&
\left( \begin{array}{cccc}
\overline{c^{*}}^{[i]}_{b_{0}0}(z) & \overline{c^{*}}^{[i]}_{b_{1}0}(z)&  \cdots & \overline{c^{*}}^{[i]}_{b_{n}0}(z) \\
\overline{c^{*}}^{[i]}_{b_{0}1}(z) & \overline{c^{*}}^{[i]}_{b_{1}1}(z)& \cdots  & \overline{c^{*}}^{[i]}_{b_{n}1}(z) \\
\vdots & \vdots & \ddots & \vdots \\
\overline{c^{*}}^{[i]}_{b_{0}n}(z) &\overline{c^{*}}^{[i]}_{b_{1}n}(z)& \cdots & \overline{c^{*}}^{[i]}_{b_{n}n}(z)
\end{array} \right)
\cdot
\left(
\begin{array}{c}
\overline{f}^{[i]}_{b_{0}}(z)\\
\overline{f}^{[i]}_{b_{1}}(z)\\
\vdots \\
\overline{f}^{[i]}_{b_{n}}(z)\\
\end{array}
\right)  \nonumber \\
&=&
\left( \begin{array}{cccc}
c^{*}_{b_{0}0}(z) & c^{*}_{b_{1}0}(z)&  \cdots & c^{*}_{b_{n}0}(z) \\
c^{*}_{b_{0}1}(z) & c^{*}_{b_{1}1}(z)& \cdots  & c^{*}_{b_{n}1}(z) \\
\vdots & \vdots & \ddots & \vdots \\
c^{*}_{b_{0}n}(z) &c^{*}_{b_{1}n}(z)& \cdots & c^{*}_{b_{n}n}(z)
\end{array} \right)
\cdot
\left(
\begin{array}{c}
\overline{f}^{[i]}_{b_{0}}(z)\\
\overline{f}^{[i]}_{b_{1}}(z)\\
\vdots \\
\overline{f}^{[i]}_{b_{n}}(z)\\
\end{array}
\right),
\end{eqnarray}
and so
\begin{eqnarray}\label{sec3aahhh}
\left( \begin{array}{ccc}
g_{0} & \cdots & g_{n} \\
\overline{g}_{0} & \cdots  & \overline{g}_{n} \\
\vdots & \ddots & \vdots \\
\overline{g}_{0}^{[n]}& \cdots & \overline{g}_{n}^{[n]}
\end{array} \right)
&=&
\left( \begin{array}{ccc}
f_{b_{0}} & \cdots & f_{b_{n}} \\
\overline{f}_{b_{0}} & \cdots  & \overline{f}_{b_{n}} \\
\vdots & \ddots & \vdots \\
\overline{f}_{b_{0}}^{[n]}& \cdots & \overline{f}_{b_{n}}^{[n]}
\end{array} \right)  \nonumber \\
& \cdot &
\left( \begin{array}{cccc}
c^{*}_{b_{0}0}(z) &  c^{*}_{b_{0}1}(z) &  \cdots & c^{*}_{b_{0}n}(z) \\
c^{*}_{b_{1}0}(z) & c^{*}_{b_{1}1}(z)& \cdots  & c^{*}_{b_{1}n}(z) \\
\vdots & \vdots & \ddots & \vdots \\
c^{*}_{b_{n}0}(z) &  c^{*}_{b_{n}1} (z)& \cdots & c^{*}_{b_{n}n}(z)
\end{array} \right),
\end{eqnarray}
implying
\begin{eqnarray}\label{sec3aah}
C(g_{0}, \cdots, g_{n})=
C(f_{b_{0}}, \ldots, f_{b_{n}})\cdot \left| \begin{array}{cccc}
c^{*}_{b_{0}0}(z) &  c^{*}_{b_{0}1}(z) &  \cdots & c^{*}_{b_{0}n}(z) \\
c^{*}_{b_{1}0}(z) & c^{*}_{b_{1}1}(z)& \cdots  & c^{*}_{b_{1}n}(z) \\
\vdots & \vdots & \ddots & \vdots \\
c^{*}_{b_{n}0}(z) &  c^{*}_{b_{n}1} (z)& \cdots & c^{*}_{b_{n}n}(z)
\end{array} \right|.
\end{eqnarray}
For simplicity,
we set
\begin{eqnarray}\label{sec3aai}
A_{b_{0}b_{1}\cdots b_{n}}(z)=\left| \begin{array}{cccc}
c^{*}_{b_{0}0}(z) &  c^{*}_{b_{0}1}(z) &  \cdots & c^{*}_{b_{0}n}(z) \\
c^{*}_{b_{1}0}(z) & c^{*}_{b_{1}1}(z)& \cdots  & c^{*}_{b_{1}n}(z) \\
\vdots & \vdots & \ddots & \vdots \\
c^{*}_{b_{n}0}(z) &  c^{*}_{b_{n}1} (z)& \cdots & c^{*}_{b_{n}n}(z)
\end{array} \right|,
\end{eqnarray}
and then we have $A_{b_{0}b_{1}\cdots b_{n}}(z)\not\equiv 0$ and $T(r,A_{b_{0}b_{1}\cdots b_{n}}(z))=o(T_{g}(r))$ from \eqref{sec3aabbbb}, \eqref{sec3aad} and \eqref{sec3aai}.

From (\ref{hab}), (\ref{sec3aah}) and (\ref{sec3aai}), we have
\begin{eqnarray}\label{habbbb}
  L&=&\frac{f_{0}f_{1}\cdots f_{q}}{A_{b_{0}b_{1}\cdots b_{n}}(z)C(f_{b_{0}},\, f_{b_{1}},\, \ldots,\, f_{b_{n}})}  = \frac{f_{a_{0}}f_{a_{1}}\cdots f_{a_{q-n-1}}}{A_{b_{0}b_{1}\cdots b_{n}}(z)H},
 \end{eqnarray}
where
\begin{eqnarray}\label{eahhhh}
H=\left| \begin{array}{ccc}
1 & \cdots & 1 \\
\overline{f}_{b_{0}}/f_{b_{0}} & \cdots  & \overline{f}_{b_{n}}/f_{b_{n}} \\
\vdots & \ddots & \vdots \\
\overline{f}_{b_{0}}^{[n]}/f_{b_{0}}& \cdots & \overline{f}_{b_{n}}^{[n]}/f_{b_{n}}
\end{array} \right|
\end{eqnarray}
and $A_{b_{0}b_{1}\cdots b_{n}}(z)$ is a meromorphic function satisfying $T(r,A_{b_{0}b_{1}\cdots b_{n}}(z))=o(T_{g}(r))$.

We set $\mathfrak{A}$ as in \eqref{A} in the proof of Theorem~\ref{lem2.1}. Let $z_{0}\in \mathbb{C}\backslash \mathfrak{A}$ be a zero of $L$ of multiplicity $\mu$. Then from
(\ref{hab}), at least one of the functions $f_{0}, f_{1}, \ldots,f_{q}$ has a zero at $z_{0}$. For each $z_{0}\in \mathbb{C}\backslash \mathfrak{A}$, from Theorem~\ref{lem2.1} we can choose
the integers $a_{0},a_{1},\ldots, a_{q-n-1}$ 
to be a
particular set of integers $0, 1, \ldots, q$
satisfying
\begin{eqnarray}\label{eawwwww}
f_{a_{0}}(z_{0})f_{a_{1}}(z_{0})\cdots f_{a_{q-n-1}}(z_{0})\neq 0,
\end{eqnarray}
and we use these integers in \eqref{habbbb}.
From  \eqref{sec3aai}, \eqref{sec3aad} and \eqref{sec3aabbbb}, we have
$A_{b_{0}\cdots b_{n}}(z_{0})\neq \infty$.
Thus combining \eqref{habbbb}, \eqref{eahhhh} with the above analysis, we get
that $z_{0}$ is a pole
of $H$ with the multiplicity at least $\mu$.
From inspection of the form of $H$ in (\ref{eahhhh}),
we deduce that the poles of $H$ arise from the zeros of $f_{b_{i}} (i=0, \ldots, n)$.
For each $j=0, \ldots, n$,
$z_{0}$ is a zero of $f_{b_{j}}$ of multiplicity $m_{b_{j}}$,
where $m_{b_{j}}\geq 0$.
Since $0$ is the  derivative-like paired value of
$f_{i}$ with the separation $c$, we have that
whenever $f_{b_{j}}(z)=0$ with the multiplicity $m_{b_{j}}$,
we get
that $\frac{f_{b_{j}}(z+ic)}{f_{b_{j}}(z)}=\infty$ with the multiplicity
$m_{b_{j}}-\max \{m_{b_{j}}-i, 0\}=\min \{i, m_{b_{j}}\}$ except for at most finitely many zeros of $f_{b_{j}}$. Hence
%
\begin{eqnarray*}
  \mu \leq \sum_{j=0} ^{n} \min \{m_{b_{j}},n\}
\end{eqnarray*}
except at most finitely many zeros of $f_{b_0}, \ldots, f_{b_n}$. Thus the conclusions hold.
\end{proof}

Theorem~\ref{thm2.1} immediately implies the following deficiency relation for derivative-like paired values of holomorphic curves.

\begin{corollary}\label{coraaa}
Under the assumption of Theorem~\ref{thm2.1}, we have
\begin{eqnarray*}
\sum_{j=0}^{q}\delta^{[n]}_{g}(0,f_{j})\leq n+1,
\end{eqnarray*}
where
\begin{eqnarray*}
\delta^{[n]}_{g}(0,f_{j})=1-\limsup_{r\to \infty}\frac{N_{n}\left(r,\frac{1}{f_{j}}\right)}{T_{g}(r)}.
\end{eqnarray*}
\end{corollary}

Theorem \ref{thm1.2} can also be used to obtain a sufficient condition, in terms of value distribution, for the growth of a holomorphic curve to be relatively fast. For this, we first need the following definition from \cite{halburdk:06AASFM}.

\begin{definition}[\cite{halburdk:06AASFM}]
We say that $a$ is an exceptional paired value of $f$ with the separation $c$ if
the following property holds for all except at most finitely many $a$-points of $f$:
Whenever $f(z)= a$ then also $f(z+c)=a$ with the same or higher multiplicity.
\end{definition}

\begin{corollary} \label{thm1.4} Let $n\geq 1$, and let $g=[g_{0}: \ldots: g_{n}]$ be a holomorphic curve of $\C$ into $\P^n(\C)$, where $g_{0},\, \ldots, g_{n}$ are linearly independent over $\mathcal{P}_{c}^{1}$.
If
    \begin{eqnarray*}
    f_{j}=\sum_{i=0}^{n}a_{ij}g_{i} \quad  j=0,\, \ldots, \, q, \, q>n,
    \end{eqnarray*}
where $a_{ij}$ are $c$-periodic entire functions satisfying $T(r,a_{ij})=o(T_{g}(r))$, such that any $n+1$ of the $q+1$ functions
 $f_{0}, \, \ldots,\, f_{q}$ are linearly independent over
  $\mathcal{P}_{c}^{1}$  and $0$ is
an exceptional paired  value of $f_{i}$ for all $i\in\{0,\ldots,q\}$, then we have $\varsigma(g) \geq 1$.
\end{corollary}

\begin{proof}[Proof of Corollary \ref{thm1.4}. ] Suppose that $\varsigma(g)<1$. Then from Theorem \ref{thm1.2}, we have that
\begin{eqnarray}\label{hacaaa}
   (q-n)T_{g}(r)\leq N\left(r,\frac{1}{L}\right)-N(r, L)+o(T_{g}(r)),
 \end{eqnarray}
where $r$ approaches infinity outside of an exceptional set $E$ of finite logarithmic measure.

But combining \eqref{eahhhh} with the assumption that $0$ is exceptional paired value of $f_{i}$, we deduce that
$H(z)$ does not have poles. Thus $L(z)$ does not have zeros on $z\in \mathbb{C}\setminus \mathfrak{A}$,
where $\mathfrak{A} \subset E$ is defined as in the proof of Theorem~\ref{lem2.1}. So we get a contradiction from \eqref{hacaaa}, and
thus we have $\varsigma(g) \geq 1$.
\end{proof}




\section{Proof of Theorem \ref{thm1.2}}\label{proofsec}

In order to prove Theorem~\ref{thm1.2}, we introduce some lemmas firstly. The following
lemma is about the growth of non-decreasing real-valued functions.

\begin{lemma}[\cite{halburdkt:14}] \label{lem2.2}
Let $T:[0,+\infty)\rightarrow [0,+\infty)$ be a non-decreasing
continuous function and let $s\in (0,\infty)$. If the hyper-order of $T$ is strictly less than one, i.e.,
\begin{eqnarray*}
  \limsup_{r\rightarrow\infty}\frac{\log \log T(r)}{\log r}=\varsigma <1
\end{eqnarray*}
and $\delta \in (0, 1-\varsigma)$ then
\begin{eqnarray*}
  T(r+s)=T(r)+o\left( \frac{T(r)}{r^{\delta}}\right)
\end{eqnarray*}
where $r$ runs to infinity outside of a set of finite logarithmic measure.
\end{lemma}

The following lemma is an extension of the analogue of the lemma on the logarithmic derivative for
finite-order meromorphic functions (\cite[Lemma 2.3]{halburdk:06JMAA}, \cite[Theorem~2.1]{halburdk:06AASFM})
to the case of functions with the hyper-order less than one.

\begin{lemma}[\cite{halburdkt:14}] \label{lem2.3}
Let $f(z)$ be a non-constant meromorphic function and $c\in \mathbb{C}$.
If $\varsigma(f)=\varsigma<1$ and $\varepsilon>0$, then
$$m\left(r,\frac{f(z+c)}{f(z)}\right)=o\left(\frac{T(r,f)}{r^{1-\varsigma-\varepsilon}}\right)$$
for all $r$ outside of a set of finite logarithmic measure.
\end{lemma}

The following lemma shows that Cartan and Nevanlinna characteristic functions are essentially the same in the one-dimensional case.

\begin{lemma}[\cite{gundersenh:04}] \label{lem2.4}
Let $h_{1}$ and $h_{2}$ be two linearly
independent entire functions that have no common zeros, and
set $f=h_{1}/h_{2}$. For positive $r$, set
\begin{eqnarray*}
  T_{f}(r)=\frac{1}{2\pi}\int_{0}^{^{2\pi}} u(re^{i\theta})d\theta-u(0),\quad
  \textrm{where}\quad u(z)=\sup \{\log|h_{1}(z)|, \log|h_{2}(z)|\}.
\end{eqnarray*}
Then
\begin{eqnarray*}
  T_{f}(r)=T(r,f)+O(1) \quad
  \textrm{as}\quad    r \to \infty.
\end{eqnarray*}
\end{lemma}

For two meromorphic functions $f$ and $g$ (where $f\not\equiv 0$ and $g\not\equiv 0$),
let $N(r,0;f,g)$ denote the counting function of the common zeros of
$f$ and $g$, counted in the following manner. If $z_{0}$ is a zero of
$f$ with multiplicity $m$ and a zero of $g$ with multiplicity $n$, then $N(r,0;f,g)$ counts
$z_{0}$ exactly $k$ times, where $k=\min \{m, n\}$.

The following lemma is a generalization of \cite[Lemma 8.1]{gundersenh:04}. The
difference is that in the original version, $f_{j}$ are linear combinations of the
functions $g_{0}, \ldots, g_{n}$ with constant coefficients, while in the following version,
we generalize the coefficients into small meromorphic functions.

\begin{lemma} \label{lem2.5}
Let $g=[g_{0}:\cdots:g_{n}]$ with $n\geq1$ be a reduced representation
of a non-constant holomorphic curve $g$.
If
\begin{eqnarray*}
      f_{j}=\sum_{i=0}^{n}a_{ij}g_{i} \quad  j=0,\, \ldots, \, q, \, q>n,
\end{eqnarray*}
where $a_{ij}$ are entire functions satisfying
$T(r,a_{ij})= o(T_{g}(r))$, such that
any $n+1$ of the $q+1$ functions $f_{0},\, \ldots,\, f_{q}$ are linearly independent over $\mathcal{P}_{c}^{1}$,
then we have
\begin{eqnarray*}
T(r,g_{j}/g_{m})+N(r,0;g_{j},g_{m} )\leq (1+o(1))T_{g}(r), \quad
  \textrm{as}\quad   r \to \infty,
\end{eqnarray*}
and for any $\mu$ and $\nu$, we have
\begin{eqnarray*}
T(r,f_{\mu}/f_{\nu})+N(r,0;f_{\mu},f_{\nu} )\leq (1+o(1))T_{g}(r), \quad
  \textrm{as}\quad   r\to \infty,
\end{eqnarray*}
where  $\mu$ and $\nu$
are distinct integers in the set $\{0, \ldots, q \}$.
\end{lemma}

\begin{proof} The proof of the first inequality is the same as the proof of
the corresponding inequality in \cite[Lemma 8.1]{gundersenh:04}. Next we prove
the second inequality. Parts of the proof are based on modifications of the ideas behind the proof of \cite[Lemma 8.1]{gundersenh:04}.
 Suppose that $f_{\mu}$ and $f_{\nu}$ are any
two distinct functions of the functions $f_{0}, f_{1}, \ldots, f_{q}$. Since
\begin{eqnarray*}
      f_{j}=\sum_{i=0}^{n}a_{ij}g_{i} \quad  j=\mu, \nu,
\end{eqnarray*}
it follows by the definition of $u(z)$ in \eqref{sec1thm1.11} that
\begin{eqnarray}\label{caaae}
c(z)&:=& \sup \{\log|f_{\mu}(z)|, \log|f_{\nu}(z)|\}= \sup \left\{\log\left|\sum_{i=0}^{n}a_{i\mu}(z)g_{i}(z)\right|, \log\left|\sum_{j=0}^{n}a_{j\nu}(z)g_{j}(z)\right|\right\} \nonumber \\
&\leq&  \sup \left\{\log \left(\sum_{i=0}^{n}|a_{i\mu}(z)|\cdot|g_{i}(z)|\right),
 \log \left(\sum_{j=0}^{n}|a_{j\nu}(z)|\cdot|g_{j}(z)|\right)\right\}\nonumber \\
&\leq& \sup \left\{ \log \left(\sum_{i=0}^{n}|a_{i\mu}(z)|e^{u(z)}\right),
 \log \left(\sum_{j=0}^{n}|a_{j\nu}(z)|e^{u(z)}\right)\right\}\nonumber \\
&\leq& u(z)+ \sup \left\{ \log \left(\sum_{i=0}^{n}|a_{i\mu}(z)|\right),
 \log \left(\sum_{j=0}^{n}|a_{j\nu}(z)|\right)\right\}\nonumber \\
&\leq& \sup_{0 \leq i \leq n} \log|g_{i}(z)| +
 \sup \left\{ \log^{+} \left(\sum_{i=0}^{n}|a_{i\mu}(z)|\right),
 \log^{+} \left(\sum_{j=0}^{n}|a_{j\nu}(z)|\right)\right\}\nonumber \\
 &\leq& \sup_{0 \leq i \leq n} \log|g_{i}(z)| +\sum_{i=0}^{n}\log^{+} |a_{i\mu}(z)|
 +\sum_{j=0}^{n}\log^{+} |a_{j\nu}(z)|+2\log(n+1),
\end{eqnarray}
whenever $z\in \mathbb{C}$.

Since $f_{\mu}$ and $f_{\nu}$ are linearly independent entire functions, there exist entire functions $h_{\mu}, h_{\nu}, \omega_{\mu\nu}$, where
$h_{\mu}, h_{\nu}$ are linearly independent and have no common zeros such that
\begin{eqnarray}\label{caaac}
     f_{\mu}=h_{\mu}\omega_{\mu\nu} \quad and \quad f_{\nu}=h_{\nu}\omega_{\mu\nu}
\end{eqnarray}
where $N(r,0, \omega_{\mu\nu})=N(r,0; f_{\mu}, f_{\nu})$.
Set
\begin{eqnarray}\label{caaaa}
     t(z)=\sup \{\log |h_{\mu}(z)|, \log |h_{\nu}(z)|\}.
\end{eqnarray}
By applying Lemma \ref{lem2.4} to $h_{\mu}$ and $h_{\nu}$, we obtain
\begin{eqnarray}\label{caaab}
     T(r, f_{\mu}/f_{\nu})=T(r,h_{\mu}/h_{\nu})=
     \frac{1}{2\pi}\int_{0}^{2\pi}t(re^{i\theta})d\theta
+O(1)   \quad \textrm{as} \quad r\to \infty.
\end{eqnarray}
From \eqref{caaac} and \eqref{caaaa}, we have
\begin{eqnarray*}
  t(z)&=&\sup\{\log|h_{\mu}(z)|, \log|h_{\nu}(z)|\} \\
  &=& \sup \left\{\log \left|\frac{f_{\mu}}{\omega_{\mu\nu} }\right|, \log\left|\frac{f_{\nu}}{\omega_{\mu\nu} }\right|\right\}\\
  &=& \sup\{\log|f_{\mu}|, \log|f_{\nu}|\}-\log|\omega_{\mu\nu}|
\end{eqnarray*}
for any $z$ satisfying $\sup\{|f_{\mu}(z)|, |f_{\nu}(z)|\}\neq 0$
and $|\omega_{\mu\nu}(z)|\neq 0$. Thus we have
\begin{eqnarray}\label{eqcaaabz}
  \frac{1}{2\pi}\int_{0}^{2\pi}t(re^{i\theta})d\theta =  \frac{1}{2\pi}\int_{0}^{2\pi}c(re^{i\theta})d\theta
  -\frac{1}{2\pi}\int_{0}^{2\pi}\log |\omega_{\mu\nu}(re^{i\theta})|d\theta
\end{eqnarray}
for those positive $r$ for which  $\sup\{|f_{\mu}(z)|, |f_{\nu}(z)|\}$ and
$|\omega_{\mu\nu}(z)|$ have no zeros on $|z|=r$.

If  $\sup\{|f_{\mu}(z)|, |f_{\nu}(z)|\}$ or
$|\omega_{\mu\nu}(z)|$ have zeros on the circle $\{z:|z|=r\}$ (where $r>0$), then following a
similar argument as in the proof of Theorem~\ref{thm1.312}, we will obtain that \eqref{eqcaaabz} holds on $|z|=r$.
Hence \eqref{eqcaaabz} holds for all positive $r$.

Following a similar method as above, we obtain from \eqref{caaae} that
\begin{eqnarray}\label{caaaeaaabbbcccddd}
\frac{1}{2\pi}\int_{0}^{2\pi} c(re^{i\theta})d\theta \leq T_{g}(r)+\sum_{i=0}^{n}m(r,a_{i\mu})
+\sum_{j=0}^{n}m(r,a_{j\nu})+O(1)
\end{eqnarray}
holds for all positive $r$. Hence, combining \eqref{caaab}, \eqref{eqcaaabz} and \eqref{caaaeaaabbbcccddd}, we obtain that
\begin{eqnarray}\label{sec3caaad}
     T(r, f_{\mu}/f_{\nu})&=& \frac{1}{2\pi}\int_{0}^{2\pi}c(re^{i\theta})d\theta
     -\frac{1}{2\pi}\int_{0}^{2\pi}\log |\omega_{\mu\nu}(re^{i\theta})|d\theta +O(1) \nonumber  \\
     &\leq& T_{g}(r) +\sum_{i=0}^{n}m(r,a_{i\mu})
+\sum_{j=0}^{n}m(r,a_{j\nu})  \nonumber\\
     &&-\frac{1}{2\pi}\int_{0}^{2\pi}
     \log |\omega_{\mu\nu}(re^{i\theta})|d\theta +O(1)  \nonumber \\
     &\leq& T_{g}(r) +\sum_{i=0}^{n}m(r,a_{i\mu})+\sum_{j=0}^{n}m(r,a_{j\nu})  \nonumber\\
     &&-\left(m(r,\omega_{\mu\nu})-m\left(r,\frac{1}{\omega_{\mu\nu}}\right)\right)+O(1)
     \quad \textrm{as} \quad r\to \infty.
\end{eqnarray}
Since $\omega_{\mu\nu}$ is entire, then from Nevanlinna's first main theorem, we get
\begin{eqnarray}\label{sec3caaae}
m(r,\omega_{\mu\nu})-m\left(r,\frac{1}{\omega_{\mu\nu}}\right)&=&T(r,\omega_{\mu\nu})-N(r,\omega_{\mu\nu})
-\left(T\left(r,\frac{1}{\omega_{\mu\nu}}\right)-N\left(r,\frac{1}{\omega_{\mu\nu}}\right)\right)\nonumber \\
&=&N\left(r,\frac{1}{\omega_{\mu\nu}}\right).
\end{eqnarray}
Since $T_{g}(r)\to \infty$ as $r\to \infty$, we have that
$
O(1)=o(T_{g}(r))
$, and
so from \eqref{sec3caaad} and \eqref{sec3caaae}, we have
\begin{eqnarray*}
T(r, f_{\mu}/f_{\nu}) &\leq& (1+o(1))T_{g}(r)-N(r,0,\omega_{\mu\nu}) \\
&=& (1+o(1))T_{g}(r)-N(r,0;f_{\mu},f_{\nu} )
 \quad as \quad  r\to \infty.
\end{eqnarray*}
Thus the conclusion holds.
\end{proof}


\begin{proof}[Proof of Theorem~\ref{thm1.2}. ] By Theorem \ref{lem2.1}, the auxiliary function
\begin{eqnarray}\label{eaee}
h(z)=\max_{\{k_{j}\}_{j=0}^{q-n-1} \subset \{ 0, \ldots, q \} } \log |f_{k_{0}}(z)\cdots f_{k_{q-n-1}}(z)|
\end{eqnarray}
gives a finite real number for all $z\in \mathbb{C}\setminus \mathfrak{A}$,
where the set $\mathfrak{A}$ is of finite logarithmic measure as defined  in Theorem \ref{lem2.1}. We let $\{a_{0},\, \ldots,\, a_{q-n-1}\}\subset \{0,\, \ldots,\, q\}$,
and
$$\{b_{0},\, \ldots,\, b_{n}\}=\{0,\, \ldots,\, q\} \setminus
\{a_{0},\, \ldots,\, a_{q-n-1}\}.$$
Then we have \eqref{sec3aaa} and \eqref{sec3aabbbb}. By using the
same method as in the proof of Theorem~\ref{thm2.1}, we have
\eqref{sec3aah} and \eqref{sec3aai}.

Next, we set the auxiliary function
\begin{eqnarray}\label{eaf}
 \widetilde{L}:=\frac{f_{0}\overline{f}_{1}\cdots \overline{f}_{n}^{[n]} f_{n+1} \cdots f_{q}}{C(g_{0},\, \cdots,\, g_{n})},
\end{eqnarray}
which is also well defined since $C(g_{0},\, \ldots,\, g_{n})\not\equiv 0$. Obviously, $\widetilde{L}$ is meromorphic.
Following the reference \cite{halburdkt:14}, we now prove
\begin{eqnarray}\label{eaffff}
N\left(r,\frac{1}{\widetilde{L}} \right)-N(r,\widetilde{L})\leq N\left(r,\frac{1}{L} \right)-N(r,L)+o(T_{g}(r)).
\end{eqnarray}
Consider first the counting functions
\begin{eqnarray}\label{eas}
N\left(r,\frac{1}{\overline{f}_{j}^{[j]}} \right) \leq N\left(r+j, \frac{1}{f_{j}}\right)
\end{eqnarray}
for $j=1,\ldots, n$. In order to apply Lemma \ref{lem2.2} to the right side of inequality
\eqref{eas}, we need to consider the growth of $N(r, 1/f_{j})$. Since
each $f_{j}$ is a linear combination of entire functions $g_{0}, \ldots, g_{n}$ with small periodic coefficients, we have
\begin{eqnarray*}\label{sec4aaa}
|f_{j}|&\leq& \sum_{i=0}^{n}|a_{ij}|\cdot|g_{i}| \nonumber\\
&\leq& \sum_{i=0}^{n} |a_{ij}|\cdot \max_{i=0,\ldots,n}| g_{i}| \nonumber\\
&\leq& \max_{i=0,\ldots,n}| g_{i}|\left(\sum_{i=0}^{n}|a_{ij}|\right).
\end{eqnarray*}
So we have
\begin{eqnarray*}
\log|f_{j}| &\leq& \log \left(  \max_{i=0,\ldots,n}|g_{i}|\left(\sum_{i=0}^{n}|a_{ij}|\right)  \right) \nonumber\\
&=& \log \left(  \max_{i=0,\ldots,n}|g_{i}|\right) +  \log \left(\sum_{i=0}^{n}|a_{ij}|\right) \nonumber\\
&=& \max_{i=0,\ldots,n}\log   |g_{i}| +  \log \left(\sum_{i=0}^{n}|a_{ij}|\right) \nonumber\\
&\leq& \sup_{i=0,\ldots,n}\log|g_{i}|+\log^{+}\left(\sum_{i=0}^{n}|a_{ij}|\right)\nonumber\\
&\leq& \sup_{i=0,\ldots,n}\log|g_{i}|+\sum_{i=0}^{n}\log^{+}|a_{ij}|+O(1).
\end{eqnarray*}

Following a similar method as in the proof of Theorem~\ref{thm1.312}, 
we get that
\begin{eqnarray*}
\frac{1}{2\pi}\int _{0}^{2\pi}\log |f_{j}(re^{i\theta})|d\theta \leq T_{g}(r)+\sum_{i=0}^{n}m(r,a_{ij})+O(1)
\end{eqnarray*}
holds for all positive $r$. By Poisson-Jensen formula we have that
\begin{eqnarray}\label{eat}
N\left(r,\frac{1}{f_{j}}\right)&=&\int_{0}^{2\pi}\log|f_{j}(re^{i\theta})|\frac{d\theta}{2\pi} \nonumber \\
&\leq&  T_{g}(r) +\sum_{i=0}^{n}m(r,a_{ij}) +O(1) \nonumber\\
&=&  (1+o(1))T_{g}(r), \quad \textrm{as} \quad r \to \infty.
\end{eqnarray}
Since $\varsigma(g)<1$,
it follows by \eqref{eat} that
\begin{eqnarray*}\label{eau}
 \delta_{j}:=\limsup_{r\rightarrow \infty}\frac{\log \log N\left(r,\frac{1}{f_{j}}\right)}{\log r}\leq \varsigma(g)<1
\end{eqnarray*}
for all $j=1,\ldots, n$. Therefore, by  Lemma \ref{lem2.2}, we have
\begin{eqnarray}\label{eav}
 N\left(r+j,\frac{1}{f_{j}}\right)
= N\left(r,\frac{1}{f_{j}}\right)+o\left(\frac{N(r,\frac{1}{f_{j}})}{r^{1-\delta_{j}-\varepsilon}} \right), \quad \textrm{as} \quad r \to \infty,
\end{eqnarray}
where $j=1, \ldots, n$ and $r$ tends to infinity outside of an exceptional set $E_{0}$ of finite logarithmic density. By using (\ref{eat}), the inequality (\ref{eav}) yields
\begin{eqnarray*}\label{eaw}
N\left(r,\frac{1}{\overline{f}_{j}^{[j]}}\right)
\leq N\left( r, \frac{1}{f_{j}}\right)+o\left(\frac{T_{g}(r)}{r^{1-\varsigma-\varepsilon}}\right),
\quad j=1, \ldots, n, \quad \textrm{as} \quad r \to \infty,
\end{eqnarray*}
outside of the exceptional set $E_{0}$ of finite logarithmic measure. Therefore,
\begin{eqnarray*}\label{eax}
N\left(r,\frac{1}{\widetilde{L}}\right)&-&N(r, \widetilde{L}) \nonumber\\
  &=& N\left(r,\frac{C(g_{0}, \ldots, g_{n})}{f_{0}\overline{f}_{1}\cdots
  \overline{f}_{n}^{[n]}f_{n+1}\cdots f_{q}}\right)-
  N\left(r,\frac{f_{0}\overline{f}_{1}\cdots
  \overline{f}_{n}^{[n]}f_{n+1}\cdots f_{q}}{C(g_{0}, \ldots, g_{n})}\right)  \nonumber\\
  &=& N\left(r,\frac{1}{f_{0}\overline{f}_{1}\cdots
  \overline{f}_{n}^{[n]}f_{n+1}\cdots f_{q}}\right)
  -
  N\left(r,\frac{1}{C(g_{0}, \ldots, g_{n})}\right)\nonumber\\
  &=&
  \sum_{j=0}^{n}N \left(r,\frac{1}{\overline{f}_{j}^{[j]}}\right)
  +N\left(r, \frac{1}{f_{n+1}\cdots f_{q}}\right)-N\left(r,
  \frac{1}{C(g_{0},\ldots,g_{n})} \right) \nonumber \\
  &\leq &  \sum_{j=0}^{n} N\left(r, \frac{1}{f_{j}}\right)
  +N\left(r, \frac{1}{f_{n+1}\cdots f_{q}}\right)\\
  &&- N\left(r,
  \frac{1}{C(g_{0},\ldots,g_{n})} \right)
  +o\left(\frac{T_{g}(r)}{r^{1-\varsigma-\varepsilon}} \right) \nonumber \\
  &=& N\left(r, \frac{1}{f_{0}\cdots f_{q}}\right)-N\left(r,
  \frac{1}{C(g_{0},\ldots,g_{n})} \right)
  +o\left(\frac{T_{g}(r)}{r^{1-\varsigma-\varepsilon}} \right)
  \nonumber \\
  &=&N\left(r,\frac{1}{L}\right)-N(r,L)+o\left(\frac{T_{g}(r)}{r^{1-\varsigma-\varepsilon}} \right),
  \end{eqnarray*}
where $r \to \infty$ outside of the exceptional set $E_{0}$ with finite logarithmic measure.

Next, we prove the inequality (\ref{hac}) for the auxiliary function $\widetilde{L}$. By substituting
(\ref{sec3aah}) (\ref{sec3aai}) into (\ref{eaf}), we have

\begin{eqnarray*}\label{eai}
\widetilde{L}&=&\frac{f_{0}\overline{f}_{1} \cdots \overline{f}_{n}^{[n]} f_{n+1} \cdots f_{q}}{ A_{b_{0}b_{1}\cdots b_{n}}(z)C(f_{b_{0}},\,f_{b_{1}},\,
\ldots,\, f_{b_{n}})}  \nonumber \\
&=& \frac{f_{0} \cdots f_{q} \cdot (\overline{f}_{1}/f_{1}) \cdots (\overline{f}_{n}^{[n]} /f_{n})}{A_{b_{0}b_{1}\cdots b_{n}}(z) C(f_{b_{0}},\, f_{b_{1}},\, \ldots,\, f_{b_{n}})} \nonumber \\
&=&  \frac{f_{b_{0}}\overline{f}_{b_{1}}\cdots \overline{f}_{b_{n}}^{[n]} \cdot f_{a_{0}}\cdots f_{a_{q-n-1}} \cdot (\overline{f}_{1} /f_{1}) \cdots (\overline{f}_{n}^{[n]}/f_{n})\cdot
(f_{b_{1}}/\overline{f}_{b_{1}})\cdots (f_{b_{n}}/\overline{f}_{b_{n}}^{[n]})}
{A_{b_{0}b_{1}\cdots b_{n}}(z)C(f_{b_{0}},\,f_{b_{1}},\,
\ldots,\, f_{b_{n}})} \nonumber \\
&=&  \frac{f_{a_{0}}\cdots f_{a_{q-n-1}} \cdot (\overline{f}_{1} /f_{1})\cdot (f_{b_{1}}/\overline{f}_{b_{1}})
\cdots (\overline{f}_{n}^{[n]}/f_{n})
\cdot (f_{b_{n}}/\overline{f}_{b_{n}}^{[n]})}
{\left( \frac{A_{b_{0}b_{1}\cdots b_{n}}(z)f_{0} \overline{f}_{0}\cdots \overline{f}_{0}^{[n]}C(f_{b_{0}}/f_{0},\, f_{b_{1}}/f_{0},\, \ldots,\, f_{b_{n}}/f_{0})}{f_{b_{0}}\overline{f}_{b_{1}}\cdots \overline{f}_{b_{n}}^{[n]}} \right)} \nonumber \\
&=&
 \frac{f_{a_{0}}\cdots f_{a_{q-n-1}} \cdot (\overline{f}_{1} /\overline{f}_{b_{1}})/(f_{1}/f_{b_{1}})
\cdots (\overline{f}_{n}^{[n]}/\overline{f}_{b_{n}}^{[n]})
/(f_{n}/f_{b_{n}})}
{\left( \frac{A_{b_{0}b_{1}\cdots b_{n}}(z)f_{0} \overline{f}_{0}\cdots \overline{f}_{0}^{[n]}C(f_{b_{0}}/f_{0},\, f_{b_{1}}/f_{0},\, \ldots,\, f_{b_{n}}/f_{0})}{f_{b_{0}}\overline{f}_{b_{1}}\cdots \overline{f}_{b_{n}}^{[n]}} \right)} \nonumber \\
&=&
 \frac{f_{a_{0}}\cdots f_{a_{q-n-1}} \cdot (\overline{f}_{1} /\overline{f}_{b_{1}})/(f_{1}/f_{b_{1}})
\cdots (\overline{f}_{n}^{[n]}/\overline{f}_{b_{n}}^{[n]})
/(f_{n}/f_{b_{n}})}
{\left( \frac{A_{b_{0}b_{1}\cdots b_{n}}(z) C(f_{b_{0}}/f_{0},\, f_{b_{1}}/f_{0},\, \ldots,\, f_{b_{n}}/f_{0})}{(f_{b_{0}}/f_{0})\cdot
(\overline{f}_{b_{1}}/\overline{f}_{0}) \cdots (\overline{f}_{b_{n}}^{[n]}/\overline{f}_{0}^{[n]})} \right)}
\end{eqnarray*}
Therefore,
\begin{eqnarray*}\label{eaj}
 \widetilde{L}=\frac{f_{a_{0}}\cdots f_{a_{q-n-1}}}{A_{b_{0}b_{1}\cdots b_{n}}(z)G(z)},
\end{eqnarray*}
where
\begin{eqnarray}\label{eak}
G(z)=\frac{\left(\frac{C(f_{b_{0}}/f_{0},\, f_{b_{1}}/f_{0},\, \ldots, \, f_{b_{n}}/f_{0})}{(f_{b_{0}}/f_{0})\cdot (\overline{f}_{b_{1}}/\overline{f}_{0})\cdots (\overline{f}_{b_{n}}^{[n]}/\overline{f}_{0}^{[n]})}\right)}
{(\overline{f}_{1}/\overline{f}_{b_{1}})/(f_{1}/f_{b_{1}})\cdots
(\overline{f}_{n}^{[n]}/\overline{f}_{b_{n}}^{[n]})/(f_{n}/f_{b_{n}})}
\end{eqnarray}
By defining
\begin{eqnarray*}\label{eal}
\omega(z)=\max _{ \{b_{j}\}_{j=0}^{n}\subset \{0,\, \ldots,\, q\}}
\log |A_{b_{0}b_{1}\cdots b_{n}}(z)G(z)|,
\end{eqnarray*}
it follows that $h(z)=\log |\widetilde{L}(z)|+\omega(z)$
for any $z\in \mathbb{C}\backslash \mathfrak{A}$ such that
 $\widetilde{L}(z)$ is non-zero and finite. Thus we have

\begin{eqnarray}\label{eam}
\int_{0}^{2\pi} h(re^{i\theta})d\theta=
\int_{0}^{2\pi}\log |\widetilde{L}(re^{i\theta})|d\theta +
\int_{0}^{2\pi}\omega (re^{i\theta})d\theta
\end{eqnarray}
for all positive $r$ outside of the set $$E_{\widetilde{L}}=\{r: z\in \mathbb{C}\backslash\mathfrak{A}, |z|=r, \widetilde{L}(z)=0\, \textrm{or}\, \widetilde{L}(z)=\infty \}.$$
By using a similar reasoning as in \cite[p. 451]{gundersenh:04} or in the
proof of Theorem~\ref{thm1.312}, it follows that \eqref{eam} holds for all positive $r$ outside of the exceptional set $E_{\mathfrak{A}}$. (Since $\widetilde L$ is meromorphic, there is a possibility of skipping this step by adding another exceptional set, according to Lemma~\ref{lem2.12}.)



Let $\{m_{0},\, \ldots,\, m_{q-n-1} \}$ be the set of indexes
for which the maximum in (\ref{eaee}) is attained for a particular choice of $z\in \mathbb{C}\setminus \mathfrak{A}$. Then by Theorem \ref{lem2.1} it follows that
\begin{equation}\label{sec3aaj}
\log |g_{j}(z)| \leq \log|f_{m_{\nu}}(z)|+\log^{+}A(z)
\end{equation}
for all $0\leq j \leq n$ and $0\leq \nu \leq q-n-1$, and so
\begin{eqnarray}\label{ean}
(q-n)T_{g}(r) &\leq& \frac{1}{2\pi}\int_{0}^{2\pi}h(re^{i\theta})d\theta+
(q-n) m(r,A(z))
\nonumber \\
&\leq& \frac{1}{2\pi}\int_{0}^{2\pi}h(re^{i\theta})d\theta+
o(T_{g}(r))
\end{eqnarray}
as $r\rightarrow \infty$ outside of the exceptional set $E_{\mathfrak{A}}$ with
finite logarithmic measure. Since the function $G$ in (\ref{eak})
contains only sums, products and quotients of fractions of the form $\overline{(f_{j}/f_{k})}^{[l]}/\overline{(f_{j}/f_{k})}^{[i]}$
where $l,i\in \{0, 1,\, \ldots,\, n\}$ satisfying $i\leq l$, and $j,k \in \{0,\, \ldots,\, q\}$, it follows by Lemmas \ref{lem2.3} and \ref{lem2.5} that
\begin{equation}\label{eao}
\frac{1}{2\pi}\int_{0}^{2\pi}\omega(re^{i\theta})d\theta \leq m(r,G) + m(r,A_{m_{q-n}m_{q-n+1}\cdots m_{q}}) = o(T_{g}(r)),
\end{equation}
as $r$ approaches infinity outside of an exceptional set $E_{1}$ of
finite logarithmic measure. Finally, by Jensen's formula,
\begin{eqnarray}\label{eaq}
\frac{1}{2\pi}\int_{0}^{2\pi} \log|\widetilde{L}(re^{i\theta})|d\theta=N\left(r,\frac{1}{\widetilde{L}}\right)
-N(r,\widetilde{L})+O(1)
\end{eqnarray}
as $r\rightarrow \infty$, and therefore by combining (\ref{eam}), (\ref{ean}), and (\ref{eaq}), we have
\begin{eqnarray}\label{ear}
(q-n)T_{g}(r)\leq N\left(r,\frac{1}{\widetilde{L}}\right)
-N(r,\widetilde{L})+o(T_{g}(r)),
\end{eqnarray}
where $r$ approaches infinity outside of the exceptional set $E_{\mathfrak{A}}\cup E_{0} \cup E_{1}$.
Since $E_{\mathfrak{A}}$, $E_{0}$ and $E_{1}$ are all of finite logarithmic measure, their union $E_{\mathfrak{A}}\cup E_{0} \cup E_{1}$ is as well.
The assertion therefore follows by substituting  inequality \eqref{eaffff} into \eqref{ear}.
\end{proof}


\def\cprime{$'$}
\providecommand{\bysame}{\leavevmode\hbox to3em{\hrulefill}\thinspace}
\providecommand{\MR}{\relax\ifhmode\unskip\space\fi MR }
\providecommand{\MRhref}[2]{%
  \href{http://www.ams.org/mathscinet-getitem?mr=#1}{#2}
}
\providecommand{\href}[2]{#2}

\end{document}